\documentclass{article}

\usepackage[english]{babel}

\usepackage[letterpaper, margin=1in]{geometry}

\usepackage{amsmath}
\usepackage{graphicx}
\usepackage{amsfonts}
\usepackage[colorlinks=true, allcolors=blue]{hyperref}

\usepackage{subfig}
\usepackage{amssymb}
\usepackage{amsthm}
\usepackage{cancel}
\usepackage{comment}
\usepackage{algorithm} 
\usepackage{algpseudocode} 
\usepackage{authblk}
\usepackage{enumitem}
\usepackage{xcolor}

\def\Op#1{ \text{Op}\left( #1 \right)}
\def\text#1{\mbox{#1}}

\def\Sym#1{ \text{Sym}\left( #1 \right)}

\newcommand{\Order}[1]{ \mathcal{O} \left( #1 \right) }

\def\SS{\mathcal{S}}

\newtheorem{theorem}{Theorem}%[section]
\newtheorem{lemma}{Lemma}%[section]
\graphicspath {{Figures/}}

\providecommand{\keywords}[1]{
\textbf{\textit{Keywords---}} #1}

\title{\textbf{A new interpolated pseudodifferential preconditioner for the Helmholtz equation in heterogeneous media}\thanks{\textbf{Funding:} The work of S. Acosta and T. Khajah was partially supported by NIH award 1R15EB035359-01A1. The work of B. Palacios was partially supported by Agencia Nacional de Investigaci\'on y Desarrollo (ANID) de Chile, Grant FONDECYT Iniciaci\'on N$^\circ$11220772. 
S. Acosta would like to thank the support provided by Texas Children’s Hospital.}}

\author[1]{Sebastian Acosta}
\author[2]{Tahsin Khajah}
\author[3]{Benjamin Palacios}

\affil[1]{\small Department of Pediatrics, Baylor College of Medicine and Texas Children's Hospital, Houston, TX, USA}

\affil[2]{\small Department of Mechanical Engineering, The University of Texas at Tyler, Tyler, TX, USA}

\affil[3]{\small Department of Mathematics, Pontificia Universidad Cat\'olica de Chile, Santiago, Chile}

\begin{document}
\maketitle

\begin{abstract}
This paper introduces a new pseudodifferential preconditioner for the Helmholtz equation in variable media with absorption. The pseudodifferential operator is associated with the multiplicative inverse to the symbol of the Helmholtz operator. This approach is well-suited for the intermediate and high-frequency regimes. The main novel idea for the fast evaluation of the preconditioner is to interpolate its symbol, not as a function of the (high-dimensional) phase-space variables, but as a function of the wave speed itself. Since the wave speed is a real-valued function, this approach allows us to interpolate in a \textit{univariate} setting even when the original problem is posed in a multidimensional physical space. As a result, the needed number of interpolation points is small, and the interpolation coefficients can be computed using the fast Fourier transform. The overall computational complexity is log-linear with respect to the degrees of freedom as inherited from the fast Fourier transform. We present some numerical experiments to illustrate the effectiveness of the preconditioner to solve the discrete Helmholtz equation using the GMRES iterative method. The implementation of an absorbing layer for scattering problems using a complex-valued wave speed is also developed. 
Limitations and possible extensions are also discussed.
\end{abstract}


%35S05  	Pseudodifferential operators as generalizations of partial differential operators
%35S15  	Boundary value problems for PDEs with pseudodifferential operators
% 35L05     Wave equation
% 41A21  	Padé approximation
% 41A28  	Simultaneous approximation

% 65N06  	Finite difference methods for boundary value problems involving PDEs

% 65T40  	Numerical methods for trigonometric approximation and interpolation

% 65D40  	Numerical approximation of high-dimensional functions; sparse grids

%%%%%%%%%%%%%%%%%%%%%%%%%%%%%%%%%%%
%%%%%%%%%  NEW SECTION  %%%%%%%%%%%
%%%%%%%%%%%%%%%%%%%%%%%%%%%%%%%%%%%

\section{Introduction}
\label{Section.Intro}

For many scientific, engineering, and biomedical applications, the numerical solution of the Helmholtz equation is essential. We are particularly motivated by ultrasound-based techniques to medical therapeutics and diagnosis where computational simulations of time-harmonic wave fields are playing an increasingly important role \cite{AmmariBook2009,AmmariBook2012,Soldati2014,Verweij2014}. Unfortunately, the discretization of the Helmholtz equation by finite difference or finite element methods renders matrices that are highly indefinite, ill-conditioned and notoriously difficult to invert, especially at high frequencies \cite{Gander2019,Ernst2012}. Mitigating these difficulties continues to be an active area of scientific research. Much of the effort has focused on designing effective and efficient preconditioners to improve the convergence of iterative solvers. Influential work includes the shifted-Laplacian preconditioner \cite{Erlangga2004,Erlangga2008,Cocquet2017}, the analytic incomplete LU preconditioner by Gander and Nataf \cite{Gander2000,Gander2005}, the sweeping preconditioners by Engquist and Ying \cite{Engquist2011a,Engquist2011} and related approaches based on domain decomposition and Schwarz-type methods \cite{Poulson2013, Stolk2013, Chen2013, Eslaminia2016, Zepeda-Nunez2016, Stolk2017, Zepeda-Nunez2018, Taus2020, Gong2021}. See a comprehensive review by Gander and Zhang \cite{Gander2019}. Other ingenious approaches include the use of controllability methods to obtain periodic solutions to the wave equation developed by Bristeau et al. \cite{Bristeau1998,GlowinskiLionsHe2008}, Heikkola et al. \cite{Heikkola2007}, Grote and Tang \cite{Grote2019}, Appelo et al. \cite{Appelo2020}, and a recent time-domain preconditioner by Stolk \cite{Stolk2021}. Also, there has been an effort to derive coercive or sign-definite formulations of the Helmholtz equation \cite{Moiola2014,Ganesh2020}.

In the present paper, we consider a Helmholtz equation of the form,
\begin{align} \label{Eqn.001}
c^{2} \Delta u + \omega^2  \, u + i \omega a \,  u = f, 
\end{align}
commonly employed to model ultrasound waves with angular frequency $\omega$, excited by a source $f$ in biological media with variable wave speed $c$ and damping coefficient $a$. We re-purpose the pseudodifferential analysis of the Helmholtz equation to design a semi-analytical matrix-free preconditioner in the form of a pseudodifferential operator
\begin{align} \label{Eqn.Intro01}
(Q v)(x) = (2 \pi)^{-d/2} \int q(x,\xi) \hat{v}(\xi) e^{i x \cdot \xi} d \xi.
\end{align}
Here $x \in \mathbb{R}^d$ is the spatial variable, $\xi \in \mathbb{R}^d$ is the frequency (Fourier dual) variable, and $\hat{v}$ is the Fourier transform of $v$. The symbol $q$ is defined in order to approximately invert the Helmholtz operator at high frequencies. The major drawback of a formulation such as \eqref{Eqn.Intro01} is that its direct evaluation is computationally expensive. Our approach to reduce the complexity of evaluating \eqref{Eqn.Intro01} is inspired by the work of Bao and Symes \cite{Bao1996} who devised an algorithm to evaluate each term of a ``classical'' pseudodifferential expansion. If the symbol $q$ is a term of a classical expansion, then by definition $q(x,\xi)$ is homogeneous in $\xi$ \cite{Taylor1991Book,Grigis1994Book}, which implies that for each fixed $x$, the function $q(x,\xi)$ is fully characterized by its profile on the unit-sphere $\xi/|\xi|$ in the frequency domain. Consequently, Bao and Symes based their algorithm on a Fourier series approximation of $q(x,\xi/|\xi|)$ which allowed them to separate the dependence on the spatial variable $x$ from the frequency variable $\xi$ in order to evaluate \eqref{Eqn.Intro01} efficiently. When $q(x,\xi)$ is not a classical term, then the separability between $x$ and $\xi$ becomes much more cumbersome to realize, especially for three-dimensional scenarios, $d=3$. In general, it is required to sample the frequency variable $\xi \in \mathbb{R}^d$ to define interpolation points and their special-function interpolants to establish the desired compression \cite{Demanet2011,Demanet2012a}. 

By contrast, we do not assume that $q$ is a general symbol or a term of a classical expansion, but instead assume that  
\begin{align} \label{Eqn.Intro02}
q(x,\xi) = q(c(x),\xi)
\end{align}
for a given function $c : \mathbb{R}^d \to [c_{\rm min} , c_{\rm max}] \subset \mathbb{R}$ that represents the wave speed in the physical domain. In other words, the pseudodifferential symbol $q$ depends on the spatial variable $x$ only through its dependence on the wave speed $c$, which is commonly the case for models of wave propagation. The range of $c$ is in $\mathbb{R}$, which opens up the possibility of using \textit{univariate} interpolation theory as the tool of approximation and compression. Specifically, let $\{c_1, c_2, ..., c_M \} \subset \mathbb{R}$ be a set of interpolation points and $\{ \varphi_1, \varphi_2, ..., \varphi_M \}$ an interpolation basis of functions such that $\varphi_{m}(c_n) = 0$ if $n \neq m$ and $\varphi_{m}(c_m) = 1$. Examples include piecewise linear functions, Lagrange polynomials, Hermite polynomials, etc. Given this interpolation setup, we define the interpolating symbol by
\begin{align}
q_M(c,\xi) = \sum_{m=1}^{M} \varphi_{m}(c)  q(c_m,\xi),
\label{Eqn.Intro03}
\end{align}
and its associated pseudodifferential operator by 
\begin{align}
(Q_M v)(x) = (2\pi)^{-d/2} \sum_{m=1}^{M}  \varphi_{m}(c(x)) \int  q(c_m,\xi) \hat{v}(\xi) e^{i x \cdot \xi} d \xi.
\label{Eqn.Intro04}
\end{align}
This construction allows us to approximate the action of the pseudodifferential operator \eqref{Eqn.Intro01} by the action of \eqref{Eqn.Intro04}, which computationally amounts to run one fast Fourier transform (FFT) of $v$ followed by $M$ inverse FFTs of the terms $q(c_m, \xi) \hat{v}(\xi)$ for $m=1,...,M$. The simple, yet powerful, assumption \eqref{Eqn.Intro02} is at the heart of our proposed fast preconditioner for the Helmholtz equation. Section \ref{Section.Prelim} briefly reviews some facts about pseudodifferential calculus, which we use in Section \ref{Section.PSD} to derive the pseudodifferential inversion of the Helmholtz operator and the proposed preconditioner. In Section \ref{Section.EvalAlgo} we describe the efficient evaluation of the preconditioner with nearly linear complexity with respect to the degrees of freedom and provide bounds on its accuracy based on univariate interpolation theory. In Section \ref{Section.NumExp} we present some numerical experiments to illustrate the effectiveness of the preconditioner to solve the Helmholtz equation using the GMRES iterative method. We illustrate the performance of the preconditioner for a range of values for the degrees of freedom, the number $M$ of interpolation points, and the frequency of the waves. 
Finally, in Section \ref{Section.Conclusion} we provide some concluding remarks and propose some areas of potential improvement.

%%%%%%%%%%%%%%%%%%%%%%%%%%%%%%%%%%%
%%%%%%%%%  NEW SECTION  %%%%%%%%%%%
%%%%%%%%%%%%%%%%%%%%%%%%%%%%%%%%%%%

\section{Preliminaries and notation}
\label{Section.Prelim}

In this section we briefly introduce the main definitions and facts about pseudodifferential calculus. Our guiding references are  \cite{Taylor1991Book,Grigis1994Book}. This section is self-contained and its notation should not be confused with notation from the rest of the paper. The first ingredient in this formulation is the Fourier transform $\mathcal{F}$ and its inverse $\mathcal{F}^{-1}$, which for an admissible function $v:  \mathbb{R}^d \to \mathbb{C}$, are respectively given by
\begin{align} \label{Eqn.FT}
\mathcal{F} v(\xi) = \hat{v}(\xi) = (2 \pi)^{-d/2} \int  v(x)  e^{-i x \cdot \xi}  dx 
\qquad \text{and} \qquad 
\mathcal{F}^{-1} \hat{v}(x) = (2 \pi)^{-d/2} \int  \hat{v}(\xi) e^{i x \cdot \xi}  d\xi.
\end{align}
A well-known fact about the Fourier transform is its relation with differentiation,
\begin{align} \label{Eqn.DTDiff}
D^{\alpha}_{x} v (x) = (2 \pi)^{-d/2} \int \xi^\alpha \hat{v}(\xi) e^{i x \cdot \xi} d\xi
\end{align}
for a multi-index $\alpha=(\alpha_1, ..., \alpha_d) \in \mathbb{N}^d$ where $D^{\alpha}_{x} = D_{1}^{\alpha_1} \dots D_{d}^{\alpha_d}$, $D_{j} = - i \partial_{x_j}$, and $\xi^\alpha =  \xi_{1}^{\alpha_1} \, \xi_{2}^{\alpha_2} \, ... \, \xi_{d}^{\alpha_d}$. For a multi-index $\alpha$, we also use the notation $|\alpha| = \alpha_1 + ... + \alpha_d$ and $\alpha ! = \alpha_1 ! \alpha_2 ! ... \alpha_d !$. 

A differential operator of order $n$ with variable coefficients $q_{\alpha}=q_{\alpha}(x)$ can be expressed as
\begin{align} \label{Eqn.PDO}
Q(x,D) = \sum_{|\alpha| \leq n} q_{\alpha}(x) D^{\alpha}_{x}.
\end{align}
Then we have
\begin{align} \label{Eqn.PDO2}
Q(x,D) v(x) = (2 \pi)^{-d/2} \int  q(x, \xi) \hat{v}(\xi) e^{i x \cdot \xi} d \xi 
\end{align}
where
\begin{align} \label{Eqn.PDO3}
 q(x, \xi) = \sum_{|\alpha| \leq n} q_{\alpha}(x)\xi^\alpha .
\end{align}
The function $q(x,\xi)$ is called the symbol of the differential operator. The representation \eqref{Eqn.PDO2} of differential operators can be used to generalize them to a larger class known as pseudodifferential operators. Note that for a differential operator, the function $q(x,\xi)$ is polynomial with respect to $\xi$. For pseudodifferential operators, the symbols are allowed to belong to larger sets, known as H\"{o}rmander's symbol classes defined as follows. Take $n \in \mathbb{R}$ (known as the order of the class) and define the symbol class $\SS^n$ to consist of all $C^\infty$ functions $q(x,\xi)$ satisfying
\begin{align} \label{Eqn.Hclass}
| D_{x}^{\beta} D_{\xi}^{\alpha} q(x,\xi)| \leq C_{\alpha \beta} (1 + |\xi|^2)^{(n - |\alpha|)/2}
\end{align}
for all multi-indices $\alpha$ and $\beta$, and some constants $C_{\alpha \beta} > 0$. For a symbol $q(x,\xi) \in \SS^n$, the associated operator $\Op{q}$, defined by
\begin{align} \label{Eqn.GeneralPSDO}
\Op{q} v(x) = (2\pi)^{-d/2} \int q(x,\xi)  \hat{v}(\xi) e^{ix\cdot \xi} d \xi,
\end{align}
is said to be a pseudodifferential operator that belongs to $\Op{\SS^n}$.
Note that by \eqref{Eqn.Hclass}, when we take a derivative of a symbol $q(x,\xi)$ with respect to $\xi$, we simply obtain another pseudodifferential operator with lower order. Given a pseudodifferential operator $P \in \Op{\SS^n}$, its symbol is denoted $\text{Sym}(P)$. Given two symbols $p(x,\xi)\in \SS^{n_1}$ and $q(x,\xi)\in \SS^{n_2}$, the composition of their respective operators has a symbol in $\SS^{n_1+n_2}$ that satisfies the following relation
\begin{align} \label{Eqn.Symbol_of_Composition}
\text{Sym}\left(\Op{p}\Op{q}\right) =  \sum_{|\alpha| \geq 0}\frac{i^{|\alpha|}}{\alpha!}D^\alpha_\xi p(x,\xi) D^\alpha_x q(x,\xi), \qquad \text{mod $\SS^{-\infty}$},
\end{align} 
where the notation \text{mod $\SS^{-\infty}$} means that the difference between the left- and right-hand sides of the equality belongs to $\SS^{-\infty} = \cap_{n \in \mathbb{N}} \SS^{-n}$. In particular, if $p \in \SS^{n}$ is an elliptic symbol bounded away from zero, then $p^{-1} \in \SS^{-n}$ and
\begin{align} \label{Eqn.Symbol_of_Inverse}
\Op{p^{-1}}\Op{p} = I + \sum_{|\alpha| \geq 1}\frac{i^{|\alpha|}}{\alpha!} 
\Op{D^\alpha_\xi p^{-1}(x,\xi) D^\alpha_x p(x,\xi)}, \qquad \text{mod $\SS^{-\infty}$},
\end{align} 
where $D_{\xi}^{\alpha} p^{-1} D_{x}^{\alpha} p \in \SS^{-|\alpha|}$. Therefore, for an elliptic symbol $p \in \SS^{n}$, $(\Op{p^{-1}}\Op{p} - I) \in \SS^{-1}$. In other words, the operator associated with the multiplicative inverse of the symbol $p$ approximates the inverse of $\Op{p}$ up to an operator of order $-1$.

Also, recall that for a pseudodifferential symbol $q \in \SS^{n}$, the following Sobolev estimate
\begin{align} 
\| \Op{q} v \|_{H^{0}(\mathbb{R}^d)} &\leq C_{00} \| v \|_{H^{n}(\mathbb{R}^d)} 
 \label{Eqn.SobNormPsiOp}
\end{align}
holds provided that $v \in H^{n}(\mathbb{R}^d)$ and for the constant $C_{00}$ appearing in \eqref{Eqn.Hclass}. This can be shown by defining
\begin{align} 
f(x) = (2 \pi)^{-d/2} \int q(x,\xi) \hat{v}(\xi) e^{i x \cdot \xi} d \xi \qquad \text{and} \qquad g(z,x) = (2 \pi)^{-d/2} \int q(z,\xi) \hat{v}(\xi) e^{i x \cdot \xi} d \xi \nonumber
\end{align}
for $v \in L^{1}(\mathbb{R}^d) \cap L^{2}(\mathbb{R}^d) \cap H^{n}(\mathbb{R}^d)$ so that $f$ and $g$ are continuous. Also notice that $f(x) = g(x,x)$ so that $|f(x)| \leq \sup_{z} |g(z,x)|$ which implies that $\| f \|_{H^{0}(\mathbb{R}^d)} \leq \| g \|_{L^{\infty}(\mathbb{R}^d; H^{0}(\mathbb{R}^d))}$. Hence, using Plancherel's theorem, the bound \eqref{Eqn.Hclass} and the equivalence of Sobolev norms in the Fourier space, we obtain
\begin{align} 
\| \Op{q} v \|_{H^{0}(\mathbb{R}^d)} = \| f \|_{H^{0}(\mathbb{R}^d)} \leq \| g \|_{L^{\infty}(\mathbb{R}^d; H^{0}(\mathbb{R}^d))} = \left(  \int \sup_{z} |q(z,\xi)|^2 |\hat{v}(\xi)|^2 d \xi \right)^{1/2} \leq C_{00} \| v \|_{H^{n}(\mathbb{R}^d)} \nonumber
\end{align}
which is extended to all $v \in H^{n}(\mathbb{R}^d)$ due to the density of $L^{1}(\mathbb{R}^d) \cap L^{2}(\mathbb{R}^d) \cap H^{n}(\mathbb{R}^d)$.

%%%%%%%%%%%%%%%%%%%%%%%%%%%%%%%%%%%
%%%%%%%%%  NEW SECTION  %%%%%%%%%%%
%%%%%%%%%%%%%%%%%%%%%%%%%%%%%%%%%%%
\section{Pseudodifferential inversion of the Helmholtz equation}
\label{Section.PSD}

We are interested in efficiently computing approximations to the solution $u$ of the Helmholtz equation \eqref{Eqn.001} subject to the Sommerfeld radiation condition. In practice, we can only approximately compute the solution $u$ restricted to a bounded domain $\Omega \subset \mathbb{R}^d$ with Lipschitz boundary $\partial \Omega$. 
In the Helmholtz equation \eqref{Eqn.001}, $\Delta = \partial_{x_1}^2 + ... + \partial_{x_d}^2$ denotes the Laplacian, $\omega$ is the angular frequency of oscillation, $c=c(x)$ is the variable wave speed, $a=a(x)$ is the damping coefficient, and $f$ is a given source with support in $\Omega$. We assume that the wave speed $c(x)$ is smooth, and bounded from above and below, that is, $0 < c_{\rm min} \leq c \leq c_{\rm max} < \infty$. Similarly, we assume that the damping coefficient $a(x)$ is a smooth function such that $0 < a_{\rm min} \leq a \leq a_{\rm max} < \infty$. We also assume that there are constant background properties $c_{o}$ and $a_{o}$ such that
$(c(x) - c_{o})$ and $(a(x) - a_o)$ are compactly supported in $\Omega$. The well-posedness of this problem in both weak and strong formulations has been established. See, for instance \cite{McLean2000,ColtonKressBook2013}.

% if $\Omega = \mathbb{R}^d$, 
% or periodic boundary conditions if $\Omega \subset \mathbb{R}^d$ is a hyper-rectangle. 

Our approach relies on using pseudodifferential calculus to approximate the inverse of the Helmholtz operator governing \eqref{Eqn.001}. The Helmholtz operator is defined by
\begin{align}
P = \omega^2 + i \omega a + c^2 \Delta  \in \Op{\SS^2}, \label{Eqn.003}
\end{align}
and its pseudodifferential symbol is given by
\begin{align}
\Sym{P} = p = \omega^2 + i \omega a - c^{2} \left( \xi_{1}^2 + ...+ \xi_{d}^2 \right)  =   \omega^2 + i \omega a - c^{2} |\xi|^2 \in \SS^2,\label{Eqn.005}
\end{align}
where the Fourier dual of the spatial variable $x = (x_{1}, ..., x_{d}) \in \mathbb{R}^d$ is denoted by $\xi = (\xi_{1}, ..., \xi_{d}) \in \mathbb{R}^d$. Due to the above assumptions on the wave speed and damping, it can be shown that the operator $P^{-1} \in \Op{\SS^{-2}}$ exists \cite{Taylor1991Book,Grigis1994Book}. Our objective is to construct an approximation for the symbol of $P^{-1}$. First, using the formula \eqref{Eqn.Symbol_of_Composition} for the symbol of the composition of the operators $P \in \Op{\SS^{2}}$ and $P^{-1} \in \Op{\SS^{-2}}$, we obtain
\begin{align}
1 = \Sym{P P^{-1}} = \sum_{|\alpha| \geq 0} \frac{i^{|\alpha|}}{\alpha !} D_{\xi}^{\alpha} \Sym{P} D_{x}^{\alpha} \Sym{P^{-1}}, \qquad \text{mod $\SS^{-\infty}$}.   \label{Eqn.009}
\end{align}
Note that the expansion on the right-hand side of \eqref{Eqn.009} contains only 3 terms because $\Sym{P}$, given by \eqref{Eqn.005}, is a second degree polynomial with respect to $\xi$. This implies that $\Sym{P^{-1}}$ satisfies the following second-order linear partial differential equation
\begin{align}
c^2 \Delta \, \Sym{P^{-1}} + 2 i c^{2} \xi \cdot \nabla \Sym{P^{-1}} + \left( \omega^2 + i \omega a - c^{2} |\xi|^2 \right) \, \Sym{P^{-1}} = 1, \qquad \text{mod $\SS^{-\infty}$},
\label{Eqn.011}
\end{align}
as a function of the space variables $x \in \mathbb{R}^d$ for each value of $\xi \in \mathbb{R}^d$.

% %%%%%%%%%%%%%%%%%%%%%%%%%%%%%%%%%%%%%

Now we propose the following expansion for the symbol of $P^{-1}$,
\begin{align}
q =  \sum_{j=2}^{+ \infty} q_{-j}, \qquad \text{mod $\SS^{-\infty}$} 
\label{Eqn.020}
\end{align}
where $q_{-j}\in \SS^{-j}$ and where the principal symbol is defined as 
\begin{align}
q_{-2} = \frac{1}{p} = \frac{1}{\omega^2 + i \omega a - c^2 |\xi|^2}.
\label{Eqn.021}
\end{align}

Plugging \eqref{Eqn.020}-\eqref{Eqn.021} into \eqref{Eqn.011} and gathering terms with the same classes, we define all of the symbols recursively
\begin{align}
q_{-3} &= - \frac{  2i c^{2} \xi \cdot \nabla    q_{-2}}{ \omega^2 + i \omega a - c^{2} |\xi|^2 } =  \frac{ 2i c^{2} \xi \cdot \left( i \omega \nabla a - |\xi|^2 \nabla c^{2} \right) }{ \left( \omega^2 + i \omega a - c^{2} |\xi|^2 \right)^3 },
\label{Eqn.023} \\
q_{-j} &= - \frac{2i c^2 \xi \cdot \nabla q_{-j+1} + c^2 \Delta q_{-j+2} }{\omega^2 + i \omega a - c^{2} |\xi|^2}, \qquad j \geq 4.
\label{Eqn.025}
\end{align}

As opposed to a classical expansion of the symbol $\Sym{P^{-1}}$, the adhoc expansion \eqref{Eqn.020} contains exactly one non-vanishing term when the acoustic medium is homogeneous. Hence, we expect a truncation of this expansion to be more accurate than a truncated classical expansion when the medium is homogeneous or when there are small variations in the wave speed $c$ and damping $a$.

The proposed preconditioner is the pseudodifferential operator associated with the principal symbol of the  expansion \eqref{Eqn.020},
\begin{align} \label{Eqn.Precond}
(Qv)(x) = (2\pi)^{-d/2} \int q_{-2}(x,\xi)  \hat{v}(\xi) e^{i x \cdot \xi} d \xi.
\end{align}

Note that the symbol $p \in \SS^{2}$ defined in \eqref{Eqn.005} is clearly elliptic \cite{Taylor1991Book,Grigis1994Book}. Therefore, following \eqref{Eqn.Symbol_of_Inverse}, we obtain that $Q P - I = R \in \Op{\SS^{-1}}$. This means that the proposed preconditioner $Q$ defined by \eqref{Eqn.Precond} approximates the inverse of the Helmholtz operator up to an error $R$ with pseudodifferential order of $-1$. By definition, such an operator possesses a symbol $r$ that decays in the frequency domain, ie., $|r| \leq C (1 + |\xi|^2)^{-1/2} \sim C/|\xi|$ as $|\xi| \to \infty$.
Consequently, we would expect the preconditioner to be more effective when restricted to problems whose solutions are highly oscillatory. We observe numerical evidence of this behavior in Section \ref{Section.NumExp} below. 

We may also characterize the performance of the preconditioner in terms of its spectral behavior. For this purpose, we consider the solutions $u$ of \eqref{Eqn.001} restricted to the bounded domain $\Omega$ using the restriction operator $\mathcal{E}_{\Omega}$ that maps $H^s(\mathbb{R}^d) \owns u \mapsto u|_{\Omega} \in H^s(\Omega)$.
Now note that since $u$ is a radiating solution to the homogeneous Helmholtz equation in $\mathbb{R}^d \setminus \Omega$ (because $f$ is supported in $\Omega$), then $Pu$ has support in $\Omega$ and $u|_{\mathbb{R}^d \setminus \Omega}$ is fully determined by $u|_{\Omega}$. Therefore, we can think of $P$ as mapping $H^s(\Omega)$ into $H^{s-2}(\Omega)$. The operator $R_{\Omega} = \mathcal{E}_{\Omega} Q P - I : H^{s}(\Omega) \to H^{s}(\Omega)$ still has pseudodifferential order of $-1$. See details in \cite[Ch. 4]{Grigis1994Book}. Hence, for any Sobolev scale $s \in \mathbb{R}$, the operator $R_{\Omega} = \mathcal{E}_{\Omega} Q P - I$ is bounded from $H^{s}(\Omega)$ to $H^{s+1}(\Omega)$ and compact from $H^{s}(\Omega)$ to $H^{s}(\Omega)$ due to the compact embedding $H^{s+1}(\Omega) \Subset H^{s}(\Omega)$ for a bounded domain $\Omega$ with Lipschitz boundary \cite[Thm. 3.27]{McLean2000}. As a direct consequence of the spectral behavior of compact operators (\cite[Thm 3.9]{Kress-Book-1999}, \cite[Corollary 2.2.13]{Drabek-Milota-2007}, \cite[Thm 6.26, Ch 3]{Kato1980}), we can make the following assertion.

\begin{theorem} \label{Thm.Spectral1}
The preconditioned operator $\mathcal{E}_{\Omega} Q P : H^{s}(\Omega) \to H^{s}(\Omega)$ is Fredholm (being a compact perturbation of the identity), and 
its spectrum consists of $z=1$ and at most a countable set of eigenvalues with no point of accumulation except, possibly, $z = 1$.
\end{theorem}

This theorem characterizes the spectral performance of the preconditioner $Q$ at the continuous level. In computational practice, both $Q$ and $P$ are discretized and it is not clear at this point how Theorem \ref{Thm.Spectral1} translates into the discrete setting. However, Theorem \ref{Thm.Spectral1} provides an insight regarding the ability of the preconditioner $Q$ to cluster the eigenvalues of $QP$, away from zero and infinity, in order to improve the convergence of the GMRES iterative method \cite[Part VI]{Trefethen2022book}.

%%%%%%%%%%%%%%%%%%%%%%%%%%%%%%%%%%%
%%%%%%%%%  NEW SECTION  %%%%%%%%%%%
%%%%%%%%%%%%%%%%%%%%%%%%%%%%%%%%%%%
\section{Evaluation algorithm}
\label{Section.EvalAlgo}

Now we describe the proposed algorithm to evaluate the pseudodifferential preconditioner defined by \eqref{Eqn.Precond}, its expected computational complexity and accuracy. We propose an algorithm for a pseudodifferential operator of the following form
\begin{align}
(Q v)(x) = (2 \pi)^{-d/2} g(x) \int q(c(x),\xi) \hat{v}(\xi) e^{i x \cdot \xi} d \xi,
\label{Eqn.101}
\end{align}
where the symbol $q$ satisfies \eqref{Eqn.Intro02}, that is, it depends on the spatial variable $x \in \mathbb{R}^d$ through the function $c=c(x)$ which plays the role of the heterogeneous wave speed in acoustic problems. For the principal symbol \eqref{Eqn.021}, this assumption requires that the damping coefficient $a=a(c(x))$ be a function of the wave speed $c$. This means that physical media with the same wave speed have the same damping coefficient as well. Also note that multiplication by $g(x)$ in \eqref{Eqn.101} does not pose any additional computational difficulties. Hence, in the sequel we set $g(x) = 1$ for simplicity.

%%%%%%%%%%%%%%%%%%%%%%%%%%%%%%%%%%%%%%%%%%
\subsection{Algorithm} \label{Subsec.Algorithm}
As stated in the previous section, it is assumed that $c=c(x)$ is a real-valued smooth function bounded from below and from above such that $0 < c_{\rm min} \leq c \leq c_{\rm max} < \infty$. Let $\{ c_{1}, c_{2}, ..., c_{M} \} \subset [c_{\rm min}, c_{\rm max}]$ be a finite ordered set of distinct interpolation points such that $c_1 = c_{\rm min}$ and $c_M = c_{\rm max}$. Also let $\{ \varphi_{1}, ..., \varphi_{M} \}$ be an interpolating basis associated with the points $\{ c_{1}, c_{2}, ..., c_{M} \}$, that is, each $\varphi_{m} : [c_{\rm min}, c_{\rm max}] \to [0,1]$ is a function such that $\varphi_{m}(c_{n}) = 0$ for all $n \neq m$, and $\varphi_{m}(c_{m}) = 1$ for each $m=1,...,M$. Using this interpolating basis, we define the pseudodifferential interpolating symbol as
\begin{align}
q_M(c,\xi) = \sum_{m=1}^{M} \varphi_{m}(c)  q(c_m,\xi),
\label{Eqn.105}
\end{align}
and its associated pseudodifferential operator as
\begin{align}
(Q_M v)(x) = (2 \pi)^{-d/2} \sum_{m=1}^{M}   \varphi_{m}(c(x)) \int q(c_m,\xi) \hat{v}(\xi) e^{i x \cdot \xi} d \xi.
\label{Eqn.103}
\end{align}

%%%%%%%%%%%%%%%%%%%%%%%%%%%%%%%%%%%%%%%%%%
\subsection{Complexity estimates} 
The complexity of the algorithm will be analyzed in terms of the number of multiplications and the discretization of the spatial domain. For simplicity, let us consider a regular $d$-dimensional grid of a hybercube containing the support of the input function $v=v(x)$. Let $N$ denote the number of grid points in each coordinate direction. Hence, there are $N^d$ discrete evaluation points known as degrees of freedom (DOF). We assume that the discrete Fourier transform and the inverse discrete Fourier transform are computed using FFT-type algorithms with one-dimensional complexity $\Order{N \log N}$. Hence, we expect the direct computational evaluation of a pseudodifferential operator of the form \eqref{Eqn.101} to have $\Order{ N^{2d} \log N}$ complexity.

For the proposed algorithm, the practical evaluation of \eqref{Eqn.103} amounts to compute one FFT of $v(x)$ followed by the sum of $M$ inverse FFTs, one for each input $q(c_m,\xi) \hat{v}(\xi)$ for $m=1,...,M$. Hence, we obtain the following result.
\begin{lemma} \label{Lemma.2}
The proposed algorithm based on \eqref{Eqn.103} to approximate the action of a pseudodifferential operator of the form \eqref{Eqn.101} has $\Order{ M N^{d} \log N}$ complexity.
\end{lemma}

Since DOF=$N^d$, the proposed algorithm is nearly linear with respect to the degrees of freedom. Hence, in general, we expect the proposed algorithm to be much more computationally efficient than the naive algorithm when $M \ll N^{d}$, provided that $M$ is sufficiently large to obtain satisfactory accuracy for the interpolation of the pseudodifferential symbol $q$. The theory of interpolation dictates that, for a well-chosen interpolation basis $\{ \varphi_{1}, ..., \varphi_{M} \}$, the number $M$ depends only on the smoothness of the symbol $q$ to render accuracy in the sup norm. Hence, when $q$ is sufficiently smooth, we can expect $M$ to be independent of $N$. This is extremely advantageous in situations where $v$ is highly oscillatory or when $q$ promotes high frequencies, thus requiring $N$ to be large to resolve such oscillations. For instance, in a typical setting for ultrasound propagation in biomedical applications, there are about $\sim 100$ wavelengths across the domain. If each wavelength is discretized with about $\sim 10$ points, then $N \sim 1000$. Therefore, DOF=$N^d \sim 10^6$ in a two-dimensional setting ($d=2$). This number of DOF is many orders of magnitude greater than the number of interpolation points $M \sim 10$ needed to resolve the profile of a pseudodifferential symbol like \eqref{Eqn.021} with reasonable accuracy. We characterize the accuracy of \eqref{Eqn.103} in more detail as follows.

%%%%%%%%%%%%%%%%%%%%%%%%%%%%%%%%%%%%%%%%%%
\subsection{Accuracy and conditioning estimates} 
\label{Subsec.Accuracy}

The accuracy of the proposed algorithm \eqref{Eqn.105}-\eqref{Eqn.103} depends on its ability to approximate the pseudodifferential symbol $q$ of the operator \eqref{Eqn.101}. A bound on the error $|q - q_M|$ can be established depending on the choice of interpolation points $\{ c_1, c_2, ..., c_M \}$ and interpolation functions $\{ \varphi_1, \varphi_2, ..., \varphi_M \}$. Here we describe error bounds for the piecewise linear interpolation strategy. For details and other polynomial strategies, see \cite[Ch. 8]{Kress1998} or \cite[Ch. 3]{Atkinson-Han-Book-2001}. The proof follows standard arguments for the real and imaginary parts of the difference $(q - q_M)$.
\begin{lemma} \label{Lemma.3}
Let $q : [c_{\rm min}, c_{\rm max}] \times \mathbb{R}^d \to \mathbb{C}$ be a smooth function such that
\begin{align}
\left| \frac{\partial^2 q}{\partial c ^2}(c,\xi) \right| \leq C (1 + |\xi|^2 )^{n/2},
\label{Eqn.109}
\end{align}
for some $n \in \mathbb{R}$, and for some constant $C>0$ independent of $\xi \in \mathbb{R}^d$. If $\{ \varphi_1, \varphi_2, ..., \varphi_M \}$ is a piecewise linear interpolation basis for the equidistant interpolation points $\{ c_1, c_2, ..., c_M \}$ with step size $h_M = (c_{\rm max} - c_{\rm min}) / (M-1)$. Then 
\begin{align}
\left| q(c,\xi) - q_M(c,\xi) \right| \leq C h_M^2  (1 + |\xi|^2 )^{n/2}.
\label{Eqn.115} 
\end{align}
\end{lemma}

As a result of this estimate, we can easily obtain the following error bound for the evaluation of the pseudodifferential operators \eqref{Eqn.101}-\eqref{Eqn.103}.

\begin{theorem} \label{Thm.Acc}
Let the symbol $q$, interpolation points $\{c_1, c_2, ..., c_M \}$, interpolation basis $\{ \varphi_1, \varphi_2, ..., \varphi_M \}$ and $h_M$ satisfy the assumptions of Lemma \ref{Lemma.3}. Then 
\begin{align}
\| (Q - Q_M)v \|_{H^{0}(\mathbb{R}^d)}  \leq  C h_M^2 \| v \|_{H^{n}(\mathbb{R}^d)}
\label{Eqn.125} 
\end{align}
for some constant $C>0$ independent of $M$.
\end{theorem}

\begin{proof}
To obtain the estimate \eqref{Eqn.125}, we simply employ the definition of the pseudodifferential operators, the equivalence of Sobolev norms in the Fourier space, and estimate \eqref{Eqn.SobNormPsiOp} and \eqref{Eqn.115} from Lemma \ref{Lemma.3} to obtain,
\begin{align}
\| (Q-Q_{M}) v \|_{H^{0}(\mathbb{R}^d)}^2
&\leq \int \sup_{x} |q(c(x),\xi) - q_{M}(c(x),\xi)|^2 |\hat{v}(\xi)|^2 d \xi \nonumber \\
&\leq (C h_{M}^2)^2 \int (1 + |\xi|^2)^{n} |\hat{v}(\xi)|^2 d \xi
\nonumber \\
&\leq (C h_{M}^2)^2 \| v \|_{H^{n}(\mathbb{R}^d)}^2
\nonumber
\end{align}
which renders the desired result for the same constant $C>0$ appearing in the assumption \eqref{Eqn.115}. 
\end{proof}

Finally, we conclude this section by characterizing the ability of the proposed preconditioner $Q_{M}$ to approximately invert the Helmholtz operator $P$. First, we verify that the assumptions of Lemma \ref{Lemma.3} are indeed satisfied by the principal symbol \eqref{Eqn.021}. By assumption, the wave speed $c$ is smooth, bounded above and below, and $(c-c_{\rm o})$ has compact support. So $c$ attains a minimum and maximum which we denote by $c_{\rm min}$ and $c_{\rm max}$, respectively. The derivatives of \eqref{Eqn.021} with respect to $c$ are as follows
\begin{align}
\frac{\partial q_{-2}}{\partial c} &= \frac{2 c |\xi|^2}{\left( \omega^2 + i \omega a - c^2 |\xi|^2 \right)^2}, \nonumber \\
\frac{\partial^2 q_{-2}}{\partial c^2} &= \frac{2 |\xi|^2 \left( \omega^2 + i \omega a - c^2 |\xi|^2 \right) + 8 c^2 |\xi|^4 }{\left( \omega^2 + i \omega a - c^2 |\xi|^2 \right)^3}. \nonumber
\end{align}
Hence $| \partial^2 q_{-2} / \partial c^2 | \leq C / (1 + |\xi|^2)$ for some constant $C>0$ that depends on $\omega$, $a$ and $c$, but not on $\xi$. Hence, \eqref{Eqn.109} is satisfied for $n=-2$ and Theorem \ref{Thm.Acc} applies to our preconditioner. 

We also have that
\begin{align} 
\mathcal{E}_{\Omega} Q_M P = \mathcal{E}_{\Omega} Q P + \mathcal{E}_{\Omega} (Q_M - Q) P = I + R_{\Omega} + \mathcal{E}_{\Omega} (Q_M - Q) P \label{Eqn.ApproxInv}
\end{align}
where the error operator $R_{\Omega} = \mathcal{E}_{\Omega} Q P - I : H^s(\Omega) \to H^s(\Omega)$ is compact and the restriction operator $\mathcal{E}_{\Omega} : H^s(\mathbb{R}^d) \to H^s(\Omega)$ is bounded. Both of these operators are described above Theorem \ref{Thm.Spectral1}. According to Theorem \ref{Thm.Acc}, we also have that
\begin{align} 
\| \mathcal{E}_{\Omega} (Q_M - Q) P v \|_{H^{0}(\Omega)} \leq \frac{C}{M^2} \| P v \|_{H^{-2}(\Omega)} \leq \frac{C}{M^2} \| v \|_{H^{0}(\Omega)}  \label{Eqn.ApproxInv2}
\end{align}
for some constant $C=C(\Omega, c,a,\omega)$ independent of $M$. Hence, the norm of $\mathcal{E}_{\Omega} (Q_M - Q) P$ can be made arbitrarily small by increasing the number of interpolation points $M \to \infty$. Therefore, the interpolated preconditioned operator $\mathcal{E}_{\Omega} Q_M P$ is a small bounded perturbation of the Fredholm operator $I + R_{\Omega}$. With these preliminaries, we are able to prove the following theorem regarding the ability of the interpolated preconditioner $Q_M$ to cluster the spectrum of the operator $\mathcal{E}_{\Omega} Q_{M} P$.

\begin{theorem} \label{Thm.Spectral2}
Let $B_{\rho}$ denote a disk in the complex plane centered at $z = 1$ with arbitrarily small radius $\rho > 0$. There exists $M=M(\rho)$ large enough such that the interpolated preconditioned operator $\mathcal{E}_{\Omega} Q_{M} P : H^{0}(\Omega) \to H^{0}(\Omega)$ is Fredholm (with the same index as $\mathcal{E}_{\Omega} Q P$), and its spectrum is contained in the disk $B_{\rho}$ except for a finite number of its eigenvalues.
\end{theorem}

\begin{proof}
This proof is a direct application of Theorem 3.16 in Chapter 4 of \cite{Kato1980}. First, we note that, as a consequence of Theorem \ref{Thm.Spectral1}, the disk $B_{\rho}$ contains all but a finite number of the eigenvalues of the Fredholm operator $\mathcal{E}_{\Omega} Q P = I + R_{\Omega}$. From \eqref{Eqn.ApproxInv}-\eqref{Eqn.ApproxInv2}, we have that 
\begin{align} 
\| \mathcal{E}_{\Omega} Q_M P - \mathcal{E}_{\Omega} Q P \|_{H^{0}(\Omega) \to H^{0}(\Omega)} \leq  \frac{C}{M^2}. \nonumber
\end{align}
Hence, for $M$ sufficiently large, we can apply \cite[Thm 5.17, Ch 4]{Kato1980} to deduce that $\mathcal{E}_{\Omega} Q_M P$ is also Fredholm (with the dame index as $\mathcal{E}_{\Omega} Q P$). Also, we apply \cite[Thm 3.16, Ch 4]{Kato1980} to assert that the disk $B_{\rho}$ contains the spectrum of $\mathcal{E}_{\Omega} Q_M P$ except for finitely many of its eigenvalues.
\end{proof}

The practical significance of Theorem \ref{Thm.Spectral2} is similar to that of Theorem \ref{Thm.Spectral1}, namely, that the interpolated preconditioner $Q_M$ possesses the ability to cluster the spectrum of $\mathcal{E}_{\Omega} Q_M P$ in the vicinity of $z = 1$ in the complex plane provided that the number of interpolation points $M$ is large enough. This property is needed for a preconditioner to accelerate the convergence of the GMRES iterative method as stated in \cite[Part VI]{Trefethen2022book}.

%%%%%%%%%%%%%%%%%%%%%%%%%%%%%%%%%%%%%%%%%%
\subsection{Absorbing layer} \label{Subsec.AbsLayer}
So far we have analyzed the proposed interpolated preconditioner $Q_M$ in the continuous setting where the Helmholtz equation is posed in all of $\mathbb{R}^d$. However, in computational practice, this unbounded domain must be truncated to a bounded domain $\Omega$ and the effect of the Sommerfeld radiation condition should be incorporated into an absorbing layer/boundary condition. Here we offer a way to incorporate an absorbing layer into the definition of the preconditioner $Q_M$. A common choice for achieving absorption is the complexification of the wave speed in the vicinity of the boundary of the truncated domain $\Omega$ \cite{Erlangga2008,Stolk2017,Acosta2024}. 
Hence, the wave speed is replaced as follows
\begin{align} 
\gamma(x) = c(x)(1 - i \zeta(x))
\label{Eqn.Complex_wave_speed}
\end{align}
where $c(x)$ is the real-valued wave speed, and $\zeta(x)$ is a real-valued, non-negative, bounded function whose support defines the absorbing layer denoted by $\Omega_{\rm AL}$. Typically, $0 \leq \zeta(x) \ll 1$. 

Recall from Subsection \ref{Subsec.Algorithm}, that because $c$ is a real-valued function, we are able to interpolate the preconditioner $Q$ over a one-dimensional interval $[c_{\rm min}, c_{\rm max}]$. We wish to preserve this \textit{univariate} interpolation property which is highly desirable from the computational point of view. Therefore, we assume that the wave speed is constant $c(x) = c_{o}$ in the absorbing layer $\Omega_{\rm AL}$. As a result, the complexified wave speed $c(x)(1 - i \zeta(x))$ has values in the complex plane 
either in the horizontal line from $c_{\rm min}$ to $c_{\rm max}$, or in the vertical line from $c_{o}$ to $c_{o}(1 + i \zeta_{\rm max})$ where $\zeta_{\rm max}$ is the maximum value of the function $\zeta$. As before, we let $\{ c_{1}, c_{2}, ..., c_{M} \} \subset [c_{\rm min}, c_{\rm max}]$ be a finite ordered set of distinct interpolation points such that $c_1 = c_{\rm min}$, $c_M = c_{\rm max}$, and there is $m_o$ such that $c_{m_o} = c_o$. Also let $\{ \varphi_{1}, ..., \varphi_{M} \}$ be an interpolating basis associated with the points $\{ c_{1}, c_{2}, ..., c_{M} \}$, that is, each $\varphi_{m} : [c_{\rm min}, c_{\rm max}] \to [0,1]$ is a function such that $\varphi_{m}(c_{n}) = 0$ for all $n \neq m$, and $\varphi_{m}(c_{m}) = 1$ for each $m=1,...,M$. 
Similarly, let $\{ \zeta_{0}, \zeta_{1}, \zeta_{2}, ..., \zeta_{\tilde{M}} \} \subset [0 , \zeta_{\rm max}]$ be a finite ordered set of distinct interpolation points such that $\zeta_0 = 0$ and $\zeta_{\tilde{M}} = \zeta_{\rm max}$ where $\tilde{M} = \tilde{M}(M)$. Also let $\{ \phi_{0}, \phi_{1}, ..., \phi_{\tilde{M}} \}$ be an interpolating basis associated with the points $\{ \zeta_{0}, \zeta_{1}, ..., \zeta_{\tilde{M}} \}$, that is, each $\phi_{m} : [- c_o \zeta_{\rm max}, 0] \to [0,1]$ is a function such that $\phi_{m}(- c_o \zeta_{n}) = 0$ for all $n \neq m$, and $\phi_{m}(- c_o \zeta_{m}) = 1$ for each $m=0,1,...,\tilde{M}$.
Using these two interpolating bases, we define the pseudodifferential interpolating symbol as
\begin{align}
q_M(\gamma,\xi) = \sum_{m=1}^{M} \varphi_{m}( \mathfrak{Re}(\gamma) )  \phi_{0}(\mathfrak{Im}(\gamma)) q(c_m,\xi) + \sum_{m=1}^{\tilde{M}} \varphi_{m_o}(\mathfrak{Re}(\gamma)) \phi_{m}( \mathfrak{Im}(\gamma) )  q( c_o (1 - i \zeta_m) ,\xi),
\label{Eqn.405}
\end{align}
and its associated pseudodifferential operator as
\begin{align}
(Q_M v)(x) &= (2 \pi)^{-d/2} \sum_{m=1}^{M}   \varphi_{m}(\mathfrak{Re}(\gamma(x))) \phi_0(\mathfrak{Im}(\gamma(x)))  \int q(c_m,\xi) \hat{v}(\xi) e^{i x \cdot \xi} d \xi \nonumber \\
& \quad  + (2 \pi)^{-d/2} \sum_{m=1}^{\tilde{M}}   \varphi_{m_o}(\mathfrak{Re}(\gamma(x)))  \phi_{m}(\mathfrak{Im}(\gamma(x))) \int q(c_o (1 - i \zeta_m),\xi) \hat{v}(\xi) e^{i x \cdot \xi} d \xi.
\label{Eqn.403}
\end{align}
where $q = q_{-2}$ is defined in \eqref{Eqn.021}.

%%%%%%%%%%%%%%%%%%%%%%%%%%%%%%%%%%%
%%%%%%%%%  NEW SECTION  %%%%%%%%%%%
%%%%%%%%%%%%%%%%%%%%%%%%%%%%%%%%%%%
\section{Numerical experiments}
\label{Section.NumExp}

In the numerical examples presented in this section, we consider a square domain $\Omega$ discretized with a regular grid with $N$ points in each coordinate direction. So there are $N^2$ evaluation points that represent the degrees of freedom. Let $P_N$ denote the second-order centered finite difference discretization of the Helmholtz operator $P$ defined by \eqref{Eqn.003} augmented by periodic boundary conditions. Also let $Q_{M,N}$ be the proposed pseudodifferential preconditioner \eqref{Eqn.103} using the symbol $q_{-2}$ from \eqref{Eqn.021}, numerically evaluated using the FFT on the same $N^2$-grid, and using $M$ interpolation terms as defined in Section \ref{Section.EvalAlgo}. So the goal is to solve the following linear system
\begin{align} \label{Eqn.MainDiscrete}
P_N u = f_N
\end{align}
where $P_N : \mathbb{C}^{N \times N} \to \mathbb{C}^{N \times N}$ and $f_N$ denotes the evaluation of \eqref{Eqn.Source} at the points on the $N^2$-grid. The preconditioned version of \eqref{Eqn.MainDiscrete} then becomes
\begin{align} \label{Eqn.MainDiscretePrecond}
(Q_{N,M} P_N) u = Q_{N,M} f_N
\end{align}
where the preconditioner $Q_{N,M}$ is described above. 

Our objective is to observe the computational complexity of the proposed preconditioner $Q_{N,M}$, its ability to precondition a discrete Helmholtz operator $P_{N}$, and observe the acceleration provided for the convergence of a GMRES solver for \eqref{Eqn.MainDiscretePrecond}. The preconditioning performance will be characterized by the following spectral condition number,
\begin{align} \label{Eqn.Cond}
\text{cond}(A) = \frac{ \max_{\lambda \in \sigma (A)} |\lambda| }{ \min_{\lambda \in \sigma (A)} |\lambda| }
\end{align}
where $\sigma(A)$ denotes the spectrum of a generic matrix $A$. Note that the spectral condition number \eqref{Eqn.Cond} characterizes the standard condition number for normal (unitarily diagonalizable) matrices such as $P_{N}$ obtained from the finite difference discretization of the Helmholtz operator. In all numerical experiments, the largest and smallest (in absolute value) eigenvalues were computed using MATLAB ``eigs" function for a matrix-free calculations using function handles.

For comparison purposes, we also define a pseudodifferential preconditioner $Q_{1}$ with a single interpolation point at the background wave speed $c_o$ as follows,
\begin{align}
(Q_1 v)(x) = (2 \pi)^{-d/2} \int \frac{\hat{v}(\xi) e^{i x \cdot \xi} }{\omega^2 + i \omega a_o - c_o^2 |\xi|^2}  d \xi,
\label{Eqn.OneInterpPoint}
\end{align}
along with its discrete counterpart $Q_{N,1}$ using the FFT for its numerical evaluation.

%%%%%%%%%%%%%%%%%%%%%%%%%%%%%%%%%%%
\subsection{Example 1: Circular inclusion}
\label{Section.Subsection1}

We consider a square domain $\Omega \subset \mathbb{R}^2$ of side $L=1$ centered at the origin. Periodic boundary conditions are assumed. 
The background medium has a wave speed $c_o = 1$. We place a circular inclusion of radius $0.05$ with wavespeed $c= c_o + \delta$. However, smooth transition is accomplished by the following definition,
\begin{align} \label{Eqn.WS}
c(x_1,x_2) = c_o + \delta H_{\eta}( 0.05 - \sqrt{ x_1^2 + x_2^2 }),
\end{align}
where $H_{\eta}(s) = 1/(1+e^{ - s / \eta})$ is a smooth version of the Heaviside function characterized by the transition parameter $\eta > 0$. The smaller $\eta$, the sharper the transition from $0$ to $1$ in $H(s)$, and consequently the sharper the transition from the background wave speed $c_o$ to the inclusion wave speed $c_o + \delta$. The source function $f$ is defined as a plane wave 
\begin{align} \label{Eqn.Source}
f(x_1,x_2) = e^{i \omega/c_o x_1}
\end{align}
with wavenumber $\omega / c_o$ where $\omega$ is the angular frequency of temporal oscillation. This experimental setup allows us to numerically test the proposed preconditioner under reproducible conditions and study its performance for a range of values for the number of interpolation points $M$, the frequency $\omega$, points per wavelength (PPW), 
the wave speed contrast parameter $\delta$, the wave speed smoothness parameter $\eta$, and the damping coefficient $a$. The points per wavelength are defined as PPW $= 2 \pi c_o N / (L \omega)$.

\paragraph{Complexity:} First, we observe the complexity of evaluating the operator $Q_{N,M}$. There results are displayed in Table \ref{tab:table1} for $N=180, 360, 720, 1440, 2880$ and $M=2, 4, 8$. Since the algorithm amounts to one FFT of the input, followed by $M$ inverse FFTs, then the assumed complexity in terms of CPU time is $T_{N} \sim M N^{d \beta} \log N$ where $d=2$ for a two-dimensional setting. Our objective is to empirically observe the order $\beta$ to numerically test the estimate obtained in Lemma \ref{Lemma.2} which dictates that $\beta=1$. An observed order of $\beta=1$ would imply that, in practice, the proposed preconditioner $Q_{N,M}$ possesses nearly linear complexity with respect to the DOF=$N^2$ in a two-dimensional setting. Table \ref{tab:table1} shows that the order $\beta$ is close to $1$ for each choice of the number of interpolation points $M$ in conformity with Lemma \ref{Lemma.2}.

\begin{table}[ht]
\centering \small
\caption {\label{tab:table1} Quantification of computational complexity through normalized CPU time for a range of grid points $N$ and interpolation nodes $M$. The model for complexity is $T_N \sim M N^{d \beta} \log N$ where $d=2$ for the two-dimensional setting. For each fixed $M$, the observed order $\beta$ is computed as the slope of a linear regression fitted through the CPU times $T_N$ versus $N$ in the log-log scale. A similar regression for all the values on the table renders an observed order $\beta =1.03$. The CPU times displayed here are normalized to the CPU time for $N=180$ and $M=2$.} 
\begin{tabular}{lrrrrrr}
\hline
\hline
$M$ & \multicolumn{5}{c}{Normalized CPU Time} &  Order $\beta$ \\
\cline{2-6} 
 &  $N = 180$  & $N = 360$  & $N = 720$ & $N = 1440$ & $N = 2880$ &  \\
\cline{2-6} 

$2$ & $1.00 \times 10^0$ & $3.23\times 10^0$ & $1.43 \times 10^1$ & $6.05 \times 10^1$ & $2.62 \times 10^2$ & $1.02$ \\
 
$4$ & $1.62 \times 10^0$ & $6.82 \times 10^0$ & $2.67 \times 10^1$ & $1.17 \times 10^2$ & $5.06 \times 10^2$ & $1.03$ \\
 
$8$ & $3.18 \times 10^0$ & $1.23 \times 10^1$ & $5.45 \times 10^1$ & $2.28 \times 10^2$ & $9.98 \times 10^2$ & $1.04$ \\

\hline
\hline
\end{tabular}
\end{table}

\paragraph{Frequency:} Here we explore the spectral behavior of the discrete operators $P_N$ and $Q_{N,M} P_{N}$, for a range of frequencies $\omega$ and interpolation points $M$, while holding the smoothness parameter $\eta = 1/800$, wave speed contrast parameter $\delta = 4$, damping coefficient $a=20$, and points per wavelength PPW = $12$ fixed. Table \ref{tab:tableFrequency} displays the results from these experiments. 
We observe that $P_{N}$ (finite difference discretization of the Helmholtz operator) is very ill-conditioned (cond($P_N$) $\sim 10^{12} - 10^{17}$) and its condition number grows with the frequency $\omega$. The preconditioner $Q_{N,1}$ with a single interpolation point for the background wave speed $c_o$ is able to reduce the condition number by several order of magnitude (cond($Q_{N,1} P_N$) $\sim 10^{7} - 10^{8}$) and behaves relatively stably with the frequency $\omega$. 
The condition number of $Q_{N,M} P_{N}$ for $M \geq 2$ is again several orders of magnitude smaller (cond($Q_{N,M} P_N$) $\sim 10^1 - 10^2$).
Additionally, we observe that the condition number of $Q_{N,M} P_{N}$ decreases as the frequency $\omega$ increases. This behavior conforms with Remark \ref{Thm.Spectral1} stating that the preconditioner $Q$ inverts the Helmholtz operator $P$ up to an error of pseudodifferential order $-1$. This means that the preconditioner becomes more effective for problems whose solutions are highly oscillatory ($\omega \to \infty$) because the error operator behaves like a low-pass filter.

\begin{table}[H]
\centering \small
\caption {\label{tab:tableFrequency} Spectral condition number \eqref{Eqn.Cond} for the discrete operators $P_{N}$ and $Q_{N,M} P_{N}$ for increasing frequencies $\omega$ and interpolation points $M$. For comparison, the preconditioner $Q_{N,1}$ with a single interpolation point is also tested. 
In all cases, the wave speed smoothness $\eta = 1/800$, wave speed contrast parameter $\delta = 4$, damping coefficient $a=20$, and points per wavelength PPW = $12$ are fixed.} 
\begin{tabular}{rlllll}
\hline
\hline
 &  $\omega = 20 \pi$  & $\omega = 40 \pi$  & $\omega = 80 \pi$ & $\omega = 160 \pi$ & $\omega = 320 \pi$  \\

\hline

$P_N$ & $7.95 \times 10^{12}$ & $1.31 \times 10^{14}$ & $2.11 \times 10^{15}$ & $3.39 \times 10^{16}$ & $5.42 \times 10^{17}$ \\

$Q_{N,1} P_N$ & $1.21 \times 10^{7}$ & $2.43 \times 10^{7}$ & $5.69 \times 10^{7}$ & $1.23 \times 10^{8}$ & $2.37 \times 10^{8}$ \\
 
$Q_{N,2} P_N$ & $7.28 \times 10^{2}$ & $4.23 \times 10^{2}$ & $1.98 \times 10^{2}$ & $7.79 \times 10^{1}$ & $4.79 \times 10^{1}$ \\
 
$Q_{N,4} P_N$ & $7.00 \times 10^{2}$ & $2.07 \times 10^{2}$ & $1.01 \times 10^{2}$ & $4.32 \times 10^{1}$ & $2.16 \times 10^{1}$ \\

$Q_{N,8} P_N$ & $5.25 \times 10^{2}$ & $1.83 \times 10^{2}$ & $7.84 \times 10^{1}$ & $3.27 \times 10^{1}$ & $1.65 \times 10^{1}$ \\

$Q_{N,16} P_N$ & $4.91 \times 10^{2}$ & $1.40 \times 10^{2}$ & $7.11 \times 10^{1}$ & $3.05 \times 10^{1}$ & $1.51 \times 10^{1}$ \\

$Q_{N,32} P_N$ & $4.91 \times 10^{2}$ & $1.32 \times 10^{2}$ & $6.97 \times 10^{1}$ & $3.17 \times 10^{1}$ & $1.35 \times 10^{1}$ \\

\hline
\hline
\end{tabular}
\end{table}

\paragraph{Points per wavelength:} Now we explore the spectral behavior of the discrete operators $P_N$ and $Q_{N,M} P_{N}$, for a range of points per wavelength PPW and interpolation points $M$, for fixed smoothness parameter $\eta = 1/800$, wave speed contrast parameter $\delta = 4$, damping coefficient $a=20$, and frequency $\omega = 80 \pi$. Table \ref{tab:tablePPW} displays the results of these experiments. We observe a rapid increase in the condition number for the unconditioned matrix $P_N$ which is tamed by the preconditioners $Q_{N,M}$. The reduction in the condition number when using a single interpolation point (at the background wave speed $c_o$) is considerable, but using $2$ or more interpolation points brings down the condition number by several orders of magnitude more until a floor is reached. We note the robustness of the preconditioner, implying that when $Q_{N,M}$ is applied to $P_N$, the condition number remains practically constant as the PPW increase.

\begin{table}[H]
\centering \small
\caption {\label{tab:tablePPW} Spectral condition number \eqref{Eqn.Cond} for the discrete operators $P_{N}$ and $Q_{N,M} P_{N}$ for 
points per wavelength PPW and interpolation points $M$. For comparison, the preconditioner $Q_{N,1}$ with a single interpolation point is also tested. 
In all cases, the wave speed smoothness $\eta = 1/800$, wave speed contrast parameter $\delta = 4$, damping coefficient $a=20$, and frequency $\omega = 80 \pi$ are fixed.} 
\begin{tabular}{rlllll}
\hline
\hline
 &  PPW $= 4$  & PPW $= 6$  & PPW $= 12$ & PPW $= 24$ & PPW $= 48$  \\

\hline

$P_N$ & $7.62 \times 10^{12}$ & $1.30 \times 10^{14}$ & $2.11 \times 10^{15}$ & $3.39 \times 10^{16}$ & $5.43 \times 10^{17}$ \\

$Q_{N,1} P_N$ & $3.74 \times 10^{7}$ & $5.02 \times 10^{7}$ & $5.69 \times 10^{7}$ & $5.90 \times 10^{7}$ & $5.96 \times 10^{7}$ \\
 
$Q_{N,2} P_N$ & $6.91 \times 10^{1}$ & $1.37 \times 10^{2}$ & $1.98 \times 10^{2}$ & $2.19 \times 10^{2}$ & $2.25 \times 10^{2}$ \\
 
$Q_{N,4} P_N$ & $4.74 \times 10^{1}$ & $8.32 \times 10^{1}$ & $1.01 \times 10^{2}$ & $1.09 \times 10^{2}$ & $1.12 \times 10^{2}$ \\

$Q_{N,8} P_N$ & $4.50 \times 10^{1}$ & $6.40 \times 10^{1}$ & $7.84 \times 10^{1}$ & $8.20 \times 10^{1}$ & $8.35 \times 10^{1}$ \\

$Q_{N,16} P_N$ & $4.12 \times 10^{1}$ & $5.90 \times 10^{1}$ & $7.11 \times 10^{1}$ & $7.62 \times 10^{1}$ & $7.79 \times 10^{1}$ \\

$Q_{N,32} P_N$ & $4.31 \times 10^{1}$ & $6.08 \times 10^{1}$ & $6.97 \times 10^{1}$ & $7.40 \times 10^{1}$ & $7.57 \times 10^{1}$ \\

\hline
\hline
\end{tabular}
\end{table}

\paragraph{Wave speed contrast:}
Now we test the performance of the preconditioner $Q_{N,M}$ under a range of values for the wave speed contrast parameter $\eta$ and its interaction with the number of interpolation points $M$. Recall that the background wave speed is $c_o = 1$ and the circular inclusion has wave speed $c = c_o + \delta$ as defined in \eqref{Eqn.WS} with a smooth transition characterized by a fixed parameter $\eta = 1/800$. The frequency $\omega = 80 \pi$, damping coefficient $a=20$, and points per wavelength PPW $=12$ are fixed too. 
Table \ref{tab:tableContrast} displays the spectral condition number for $Q_{N,M} P_{N}$ for these experiments. First, we observe that the larger the contrast in the wave speed, the larger the condition number of the discrete Helmholtz operator $P_N$, and the more need there is for the preconditioner. We also note that, for each fixed contrast parameter $\delta$, the preconditioner improves as the number of interpolation points $M$ increases but a floor on its performance is quickly reached.  

\begin{table}[H]
\centering \small
\caption {\label{tab:tableContrast} Spectral condition number \eqref{Eqn.Cond} for the discrete operators $P_{N}$ and $Q_{N,M} P_{N}$ for various numbers of interpolation points $M$ and wave speed contrast parameters $\delta$. For comparison, the preconditioner $Q_{N,1}$ with a single interpolation point is also tested. 
In all cases, the wave speed smoothness $\eta = 1/800$, frequency $\omega = 80 \pi$, damping coefficient $a=20$, and points per wavelength PPW = $12$ are fixed.} 
\begin{tabular}{rlllll}
\hline
\hline
 &  $\delta = 1$  & $\delta = 2$  & $\delta = 4$ & $\delta = 8$ & $\delta = 16$  \\

\hline

$P_N$ & $5.33 \times 10^{13}$ & $2.72 \times 10^{14}$ & $2.11 \times 10^{15}$ & $2.22 \times 10^{16}$ & $2.83 \times 10^{17}$ \\

$Q_{N,1} P_N$ & $9.70 \times 10^5$ & $6.47\times 10^6$ & $5.69 \times 10^7$ & $6.28 \times 10^8$ & $8.14 \times 10^9$ \\
 
$Q_{N,2} P_N$ & $9.54 \times 10^0$ & $3.22\times 10^1$ & $1.98 \times 10^2$ & $1.90 \times 10^3$ & $2.37 \times 10^4$ \\
 
$Q_{N,4} P_N$ & $1.01 \times 10^1$ & $2.60\times 10^1$ & $1.01 \times 10^2$ & $5.59 \times 10^2$ & $6.31 \times 10^3$ \\

$Q_{N,8} P_N$ & $9.86 \times 10^0$ & $2.35\times 10^1$ & $7.84 \times 10^1$ & $2.71 \times 10^2$ & $1.55 \times 10^3$ \\

$Q_{N,16} P_N$ & $9.66 \times 10^0$ & $2.34\times 10^1$ & $7.11 \times 10^1$ & $2.36 \times 10^2$ & $1.08 \times 10^3$ \\

$Q_{N,32} P_N$ & $9.64 \times 10^0$ & $2.33\times 10^1$ & $6.97 \times 10^1$ & $2.59 \times 10^2$ & $1.06 \times 10^3$ \\

\hline
\hline
\end{tabular}
\end{table}

\paragraph{Wave speed smoothness:} Here we test the performance of the preconditioner $Q_{N,M}$ for a range of values for the smoothness parameter $\eta$ and its interaction with the number of interpolation points $M$. The other parameters $\delta = 4$, $\omega = 80 \pi$, $a=20$, and PPW $= 12$ are fixed. The results from these experiments are shown in Table \ref{tab:tableSmoothness}. As expected from the mathematical theory of pseudodifferential operators and the accuracy of interpolation, we observe that the proposed preconditioner $Q_{N,M}$ is more effective when the wave speed profile is smoother. As before, when using a single interpolation point (at the background wave speed $c_o$), the reduction in the condition number is significant, but using $2$ or more interpolation points brings down the condition number by several orders of magnitude more.

\begin{table}[H]
\centering \small
\caption {\label{tab:tableSmoothness} Spectral condition number \eqref{Eqn.Cond} for the discrete operators $P_{N}$ and $Q_{N,M} P_{N}$ for various numbers of interpolation points $M$ and wave speed smoothness parameters $\eta$. For comparison, the preconditioner $Q_{N,1}$ with a single interpolation point is also tested. In all cases, the wave speed contrast $\delta = 4$, frequency $\omega = 80 \pi$, damping coefficient $a=20$, and points per wavelength PPW = $12$ are fixed.} 
\begin{tabular}{rlllll}
\hline
\hline
 &  $\eta = 1/200$  & $\eta = 1/400$  & $\eta = 1/800$ & $\eta = 1/1600$ & $\eta = 1/3200$  \\

\hline

$P_N$ & $2.10 \times 10^{15}$ & $2.11 \times 10^{15}$ & $2.11 \times 10^{15}$ & $2.11 \times 10^{15}$ & $2.11 \times 10^{15}$ \\

$Q_{N,1} P_N$ & $2.09 \times 10^{7}$ & $4.11 \times 10^{7}$ & $5.69 \times 10^{7}$ & $6.44 \times 10^{7}$ & $6.72 \times 10^{7}$ \\
 
$Q_{N,2} P_N$ & $8.15 \times 10^{1}$ & $6.87 \times 10^{1}$ & $1.98 \times 10^{2}$ & $6.87 \times 10^{2}$ & $1.76 \times 10^{3}$ \\
 
$Q_{N,4} P_N$ & $7.93 \times 10^{0}$ & $3.08 \times 10^{1}$ & $1.01 \times 10^{2}$ & $3.03 \times 10^{2}$ & $9.80 \times 10^{2}$ \\

$Q_{N,8} P_N$ & $9.55 \times 10^{0}$ & $2.43 \times 10^{1}$ & $7.84 \times 10^{1}$ & $2.13 \times 10^{2}$ & $5.77 \times 10^{2}$ \\

$Q_{N,16} P_N$ & $9.41 \times 10^{0}$ & $2.36 \times 10^{1}$ & $7.11 \times 10^{1}$ & $2.24 \times 10^{2}$ & $5.83 \times 10^{2}$ \\

$Q_{N,32} P_N$ & $8.61 \times 10^{0}$ & $2.50 \times 10^{1}$ & $6.97 \times 10^{1}$ & $1.92 \times 10^{2}$ & $4.95 \times 10^{2}$ \\

\hline
\hline
\end{tabular}
\end{table}

\paragraph{Damping:} We also test the performance of the preconditioner $Q_{N,M}$ for a range of values for the damping coefficient $a$ and number of interpolation points $M$. The other parameters $\delta = 4$, $\eta = 1/800$, $\omega = 80 \pi$, and PPW $= 12$ are fixed. The results from these experiments are displayed in Table \ref{tab:tableDamping}. We observe that the smaller the damping coefficient, the less effective the preconditioner $Q_{N,M}$. However, over this range of value for the damping coefficient, the preconditioner seems to be quite robust.

\begin{table}[h]
\centering \small
\caption {\label{tab:tableDamping} Spectral condition number \eqref{Eqn.Cond} for the discrete operators $P_{N}$ and $Q_{N,M} P_{N}$ for various interpolation points $M$ and damping coefficients $a$. For comparison, the preconditioner $Q_{N,1}$ with a single interpolation point is also tested. In all cases, the wave speed contrast $\delta = 4$, wave speed smoothness $\eta = 1/800$, frequency $\omega = 80 \pi$, and points per wavelength PPW = $12$ are fixed.} 
\begin{tabular}{rlllll}
\hline
\hline
 &  $a = 5$  & $a = 10$  & $a = 20$ & $a = 40$ & $a = 80$  \\

\hline

$P_N$ & $2.11 \times 10^{15}$ & $2.11 \times 10^{15}$ & $2.11 \times 10^{15}$ & $2.11 \times 10^{15}$ & $2.11 \times 10^{15}$ \\

$Q_{N,1} P_N$ & $8.92 \times 10^{7}$ & $6.83 \times 10^{7}$ & $5.69 \times 10^{7}$ & $4.12 \times 10^{7}$ & $2.47 \times 10^{7}$ \\
 
$Q_{N,2} P_N$ & $2.74 \times 10^{2}$ & $2.42 \times 10^{2}$ & $1.98 \times 10^{2}$ & $1.41 \times 10^{2}$ & $8.84 \times 10^{1}$ \\
 
$Q_{N,4} P_N$ & $1.69 \times 10^{2}$ & $1.25 \times 10^{2}$ & $1.01 \times 10^{2}$ & $8.09 \times 10^{1}$ & $5.81 \times 10^{1}$ \\

$Q_{N,8} P_N$ & $9.99 \times 10^{1}$ & $8.40 \times 10^{1}$ & $7.84 \times 10^{1}$ & $6.61 \times 10^{1}$ & $4.81 \times 10^{1}$ \\

$Q_{N,16} P_N$ & $8.49 \times 10^{1}$ & $7.81 \times 10^{1}$ & $7.11 \times 10^{1}$ & $6.07 \times 10^{1}$ & $4.62 \times 10^{1}$ \\

$Q_{N,32} P_N$ & $8.30 \times 10^{1}$ & $7.54 \times 10^{1}$ & $6.97 \times 10^{1}$ & $6.08 \times 10^{1}$ & $4.61 \times 10^{1}$ \\

\hline
\hline
\end{tabular}
\end{table}

\paragraph{GMRES implementation:} Finally we also report some numerical results for solving the Helmholtz equation with and without the proposed preconditioner \eqref{Eqn.103} using the GMRES iterative method. MATLAB ``gmres'' function was employed for matrix-free calculations using function handles. The restart parameter was set equal to $10$. For the problem described by \eqref{Eqn.WS}-\eqref{Eqn.Source}, the evolution of the residuals over the GMRES iterations are illustrated in Figure \ref{Fig.Solution_u} for frequencies $\omega = 400 \pi$, $\omega = 600 \pi$ and $\omega = 800 \pi$, which respectively corresponds to $200$, $300$ and $400$ wavelengths across the domain. In all cases, the damping coefficient $a=20$, the wave speed contrast $\delta = 1$, the wave speed smoothness $\eta = 1/800$, and
the number of interpolation points is $M=8$ are fixed. The number of points per wavelength is PPW = $12$ which leads to $N=2400$, $N=3600$ and $N=4800$, respectively. Hence, DOF = $5.760 \times 10^6$, DOF = $1.296 \times 10^7$ and DOF = $2.304 \times 10^7$, respectively.

We observe from Figure \ref{Fig.Solution_u} that applying the GMRES method naively to \eqref{Eqn.MainDiscrete} leads to a very slow convergence pattern. By contrast, the application of the GMRES method to the preconditioned system \eqref{Eqn.MainDiscretePrecond} leads to much faster convergence. In fact, at $100$ iterations, the residual is at least $7$ orders of magnitude smaller when the proposed preconditioner is employed. For reference, the amplitude of the numerical solutions obtained after $100$ iterations are displayed in Figure \ref{Fig.Solution_u}.

\begin{figure}[H]
\centering
\captionsetup{font=small}
\subfloat[Amplitude of the solution for $\omega = 400 \pi$]{
\includegraphics[height=0.38 \textwidth, trim = 0 0 0 0, clip]{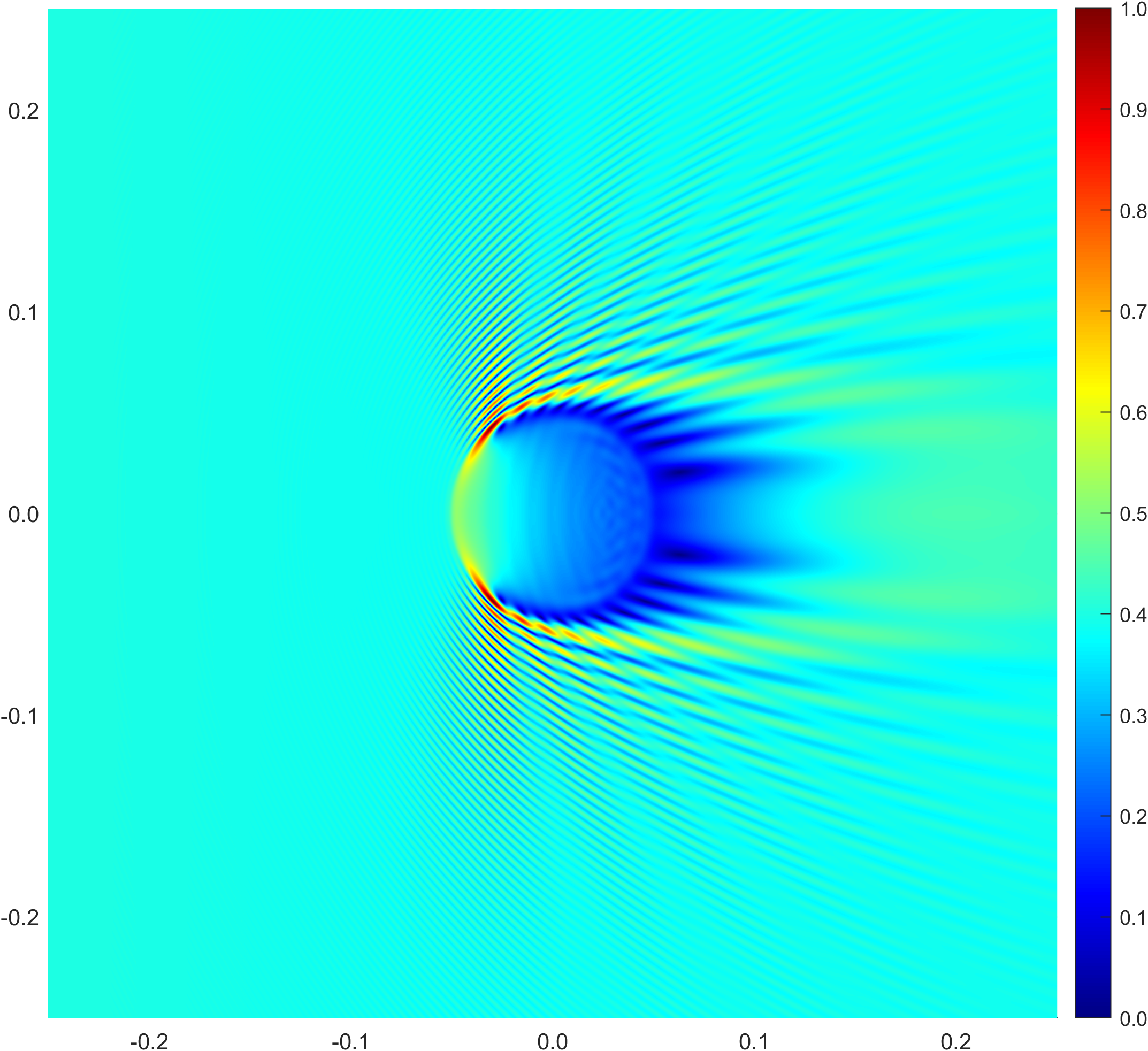}
\label{subfig:Ampl_200}
}
\subfloat[GMRES residuals for $\omega = 400 \pi$]{
\includegraphics[height=0.38 \textwidth, trim = 0 0 0 0, clip]{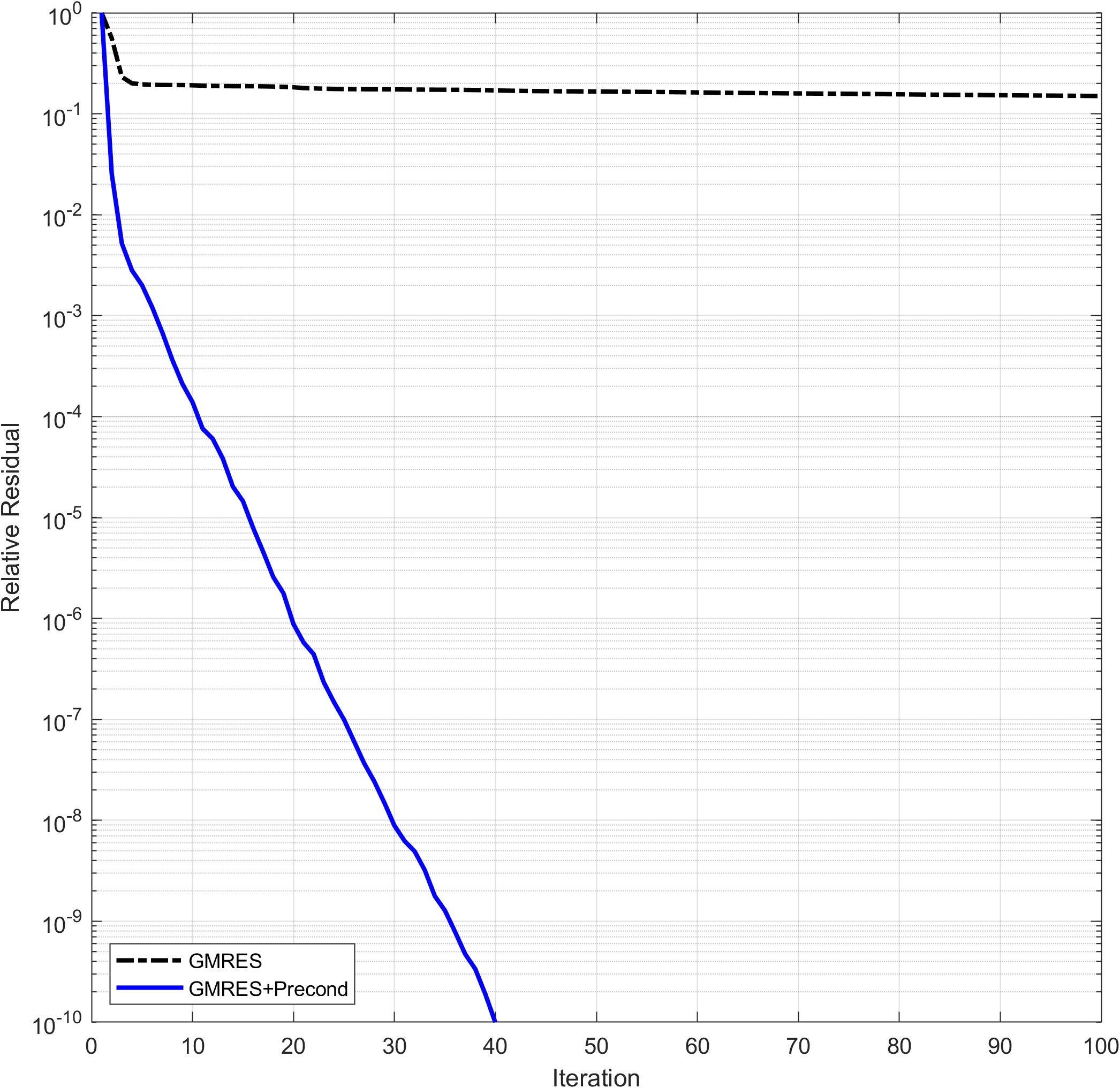}
\label{subfig:GMRES_200}
} \\
\subfloat[Amplitude of the solution for $\omega = 600 \pi$]{
\includegraphics[height=0.38 \textwidth, trim = 0 0 0 0, clip]{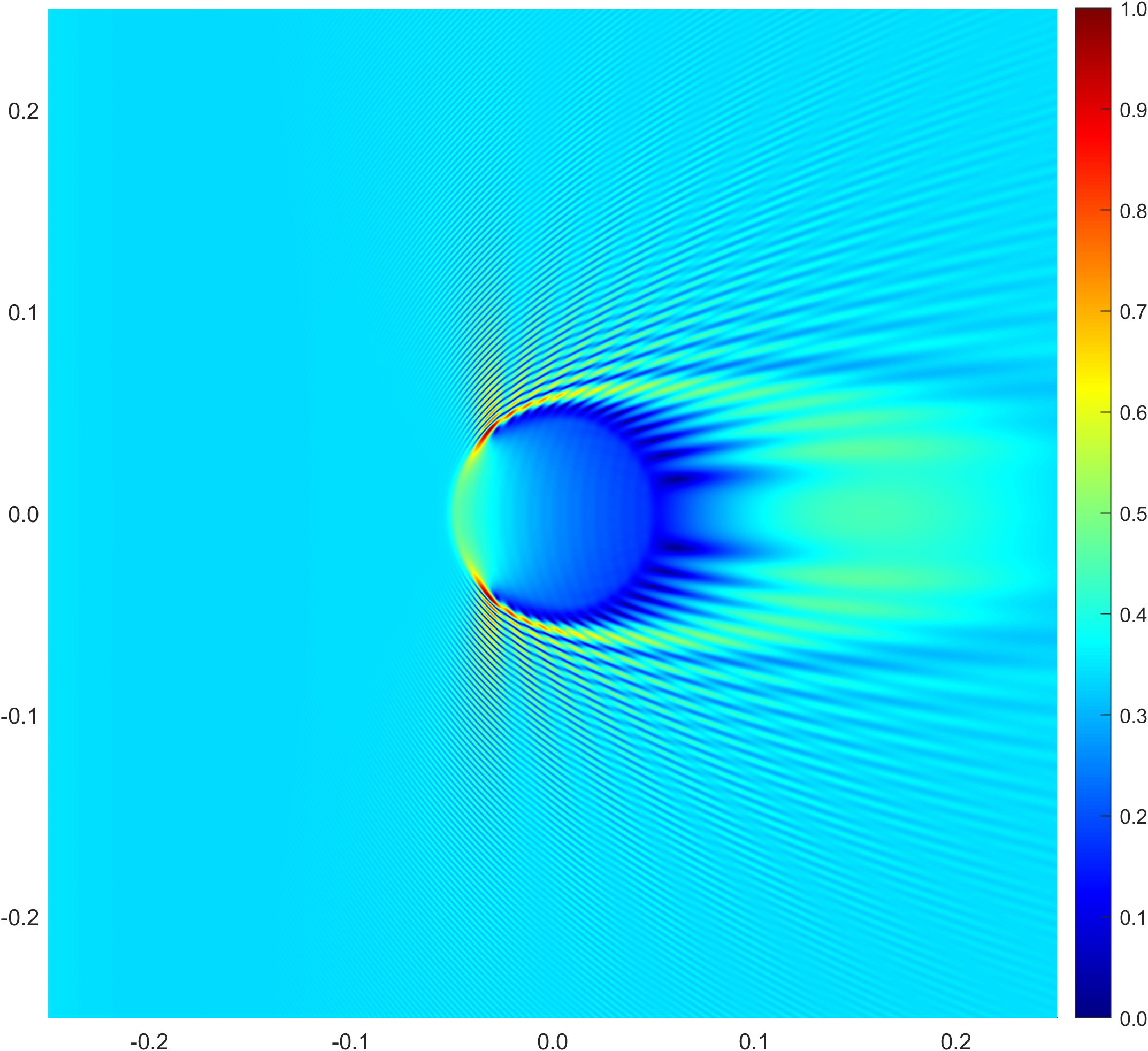}
\label{subfig:Ampl_300}
}
\subfloat[GMRES residuals for $\omega = 600 \pi$]{
\includegraphics[height=0.38 \textwidth, trim = 0 0 0 0, clip]{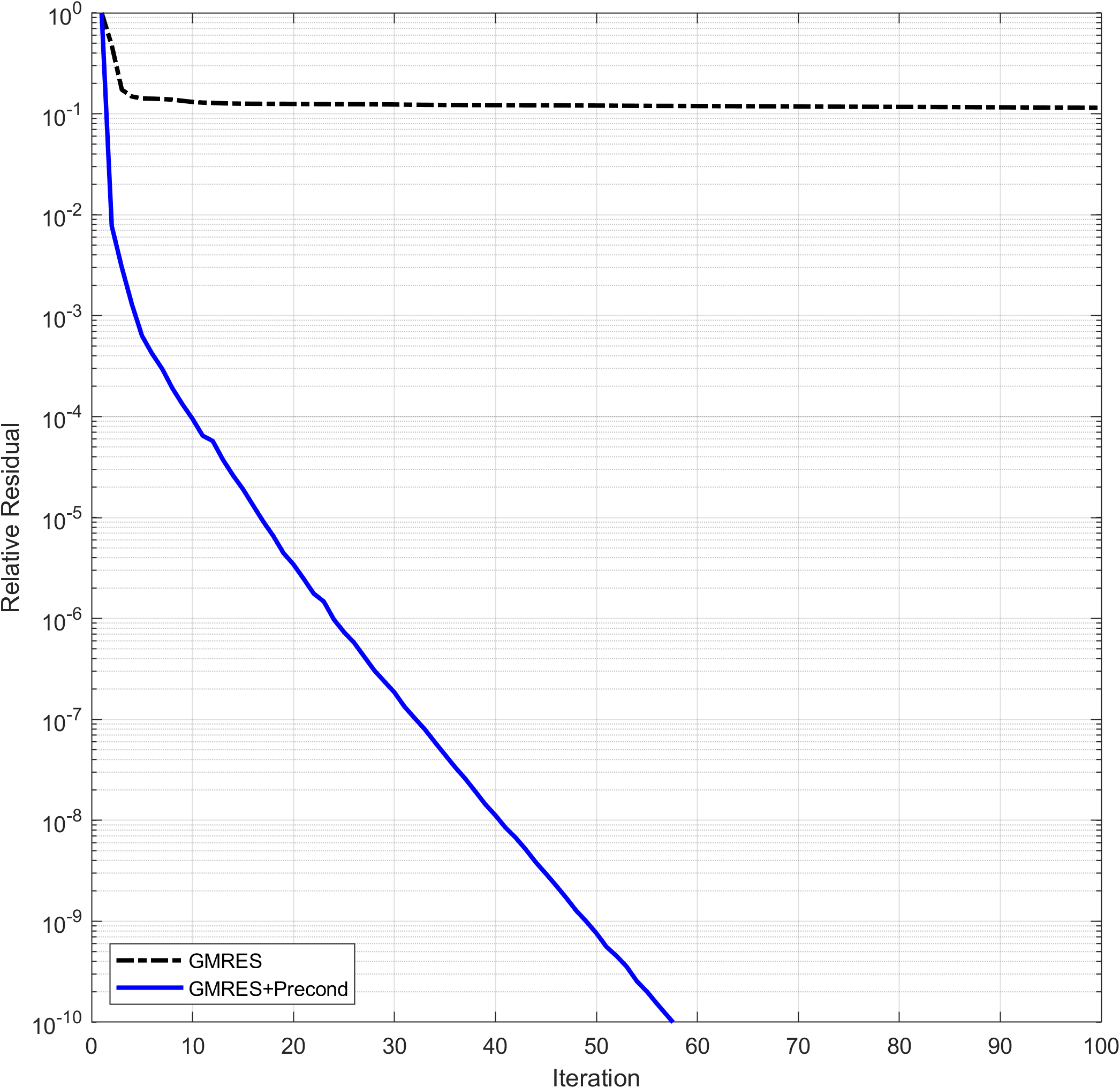}
\label{subfig:GMRES_300}
} \\
\subfloat[Amplitude of the solution for $\omega = 800 \pi$]{
\includegraphics[height=0.38 \textwidth, trim = 0 0 0 0, clip]{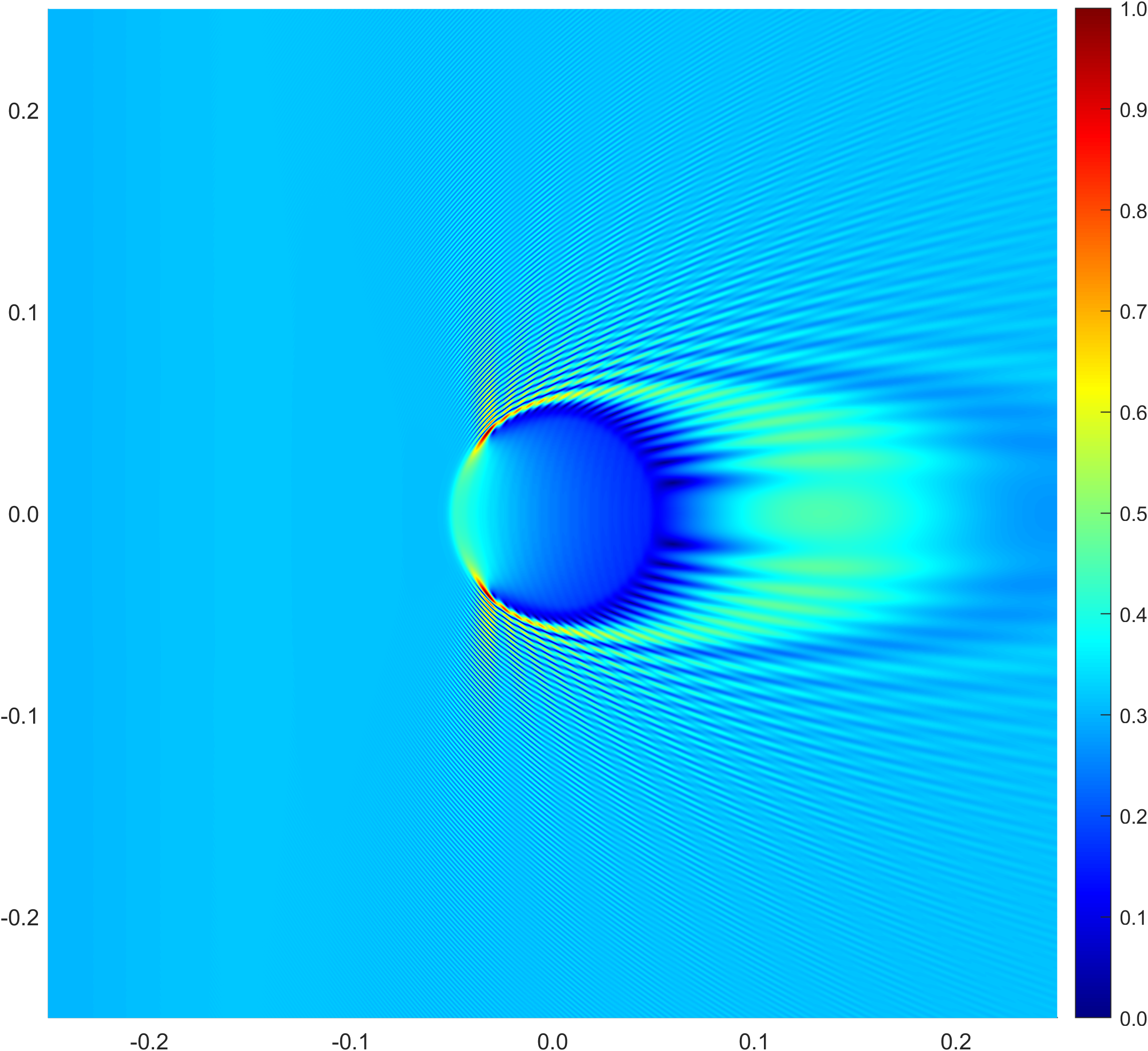}
\label{subfig:Ampl_400}
}
\subfloat[GMRES residuals for $\omega = 800 \pi$]{
\includegraphics[height=0.38 \textwidth, trim = 0 0 0 0, clip]{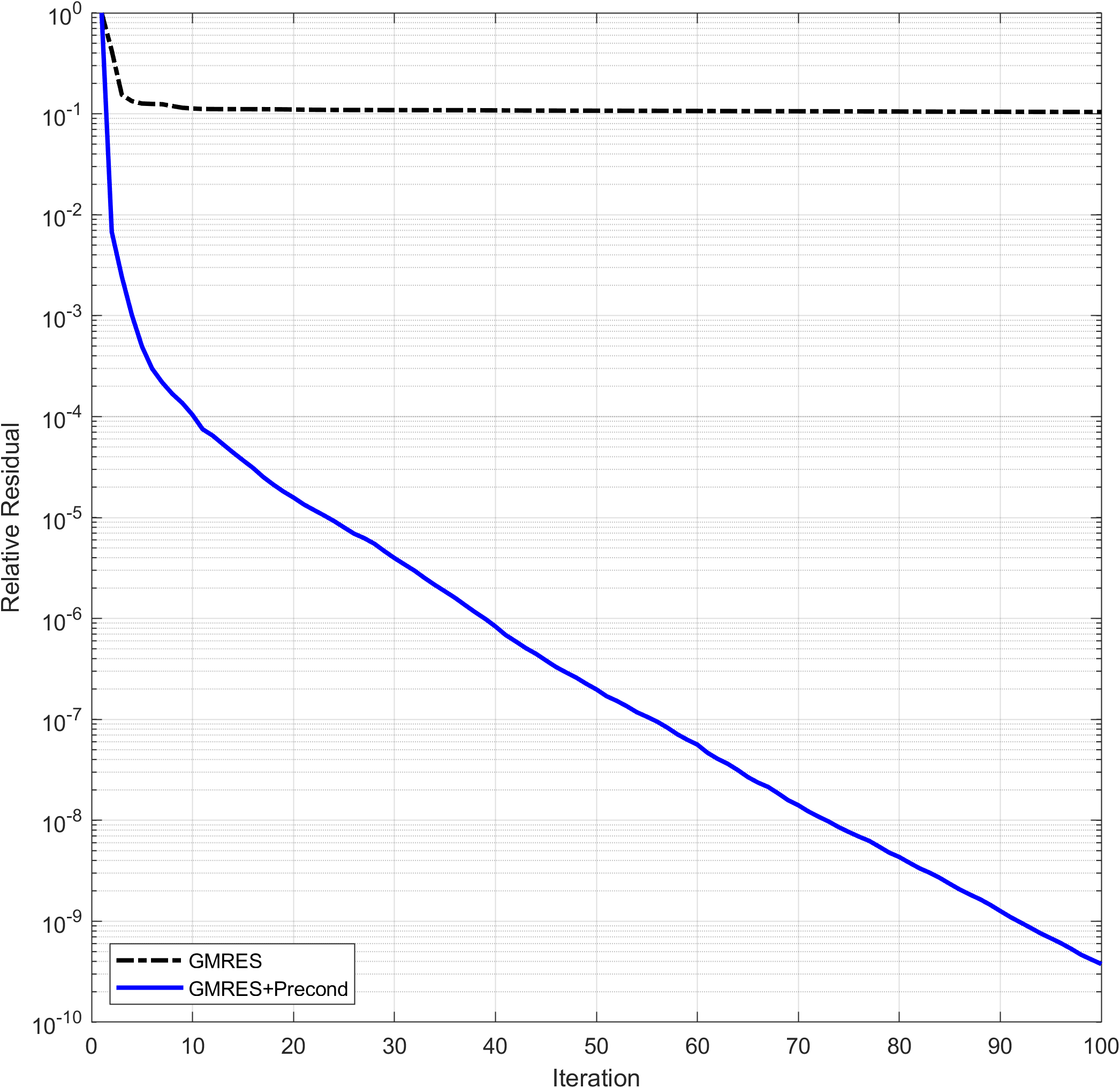}
\label{subfig:GMRES_400}
}
\caption{Amplitude of the solution $u$ (normalized in the $L^{\infty}$-norm) and relative residual for the first 100 iterations of the GMRES method applied to \eqref{Eqn.MainDiscrete} and to \eqref{Eqn.MainDiscretePrecond} for frequencies $\omega=400 \pi$, $\omega = 600 \pi$ and $\omega=800 \pi$, respectively. In all cases, the GMRES+Predonditioner method reaches residuals more than $7$ orders of magnitude smaller than the GMRES method at the $100$th iteration. The wave speed profile $c$ and source $f$ are defined by \eqref{Eqn.WS} (using $\delta = 1$ and $\eta = 1/800$) and \eqref{Eqn.Source}, respectively.}
\label{Fig.Solution_u}
\end{figure}

\bigskip

\subsection{Example 2: Human head phantom and absorbing layer}
\label{Section.Subsection2}

In this subsection we apply the proposed preconditioner to solve the Helmholtz equation in a domain whose acoustic properties conform to those of a human head as defined by a synthetic phantom available as supplemental material from \cite{Aubry2022}. The spatial dimensions of the phantom are scaled in order to fit in a square domain $\Omega \subset \mathbb{R}^2$ of side $L=1$ centered at the origin. Periodic boundary conditions are assumed at the sides of the square $\Omega$, but an absorbing layer is introduced to enforce the outgoing behavior of the waves. This absorbing layer follows the setup described in Subsection \ref{Subsec.AbsLayer}. Specifically, the function $\zeta$ defining the absorbing layer is given by
\begin{align} \label{Eqn.Zeta}
\zeta(x_1,x_2) = \begin{cases}
			0, & \text{if $\sqrt{x_1^2 + x_2^2} < 0.375 $,}\\
            0.4 (\sqrt{x_1^2 + x_2^2} - 0.375), & \text{if $\sqrt{x_1^2 + x_2^2} \geq 0.375 $.}
		 \end{cases}
\end{align}
The background medium has a wave speed $c_o = 1$ and damping coefficient $a_{o} = 20.9$. The human head phantom also contains a cross section of the skull whose wave speed is $c_{\rm skull} = 6.71$ and damping coefficient $a_{\rm skull} = 258$. The soft tissues representing the brain matter and fluids contained by the skull are assumed to have a uniform wave speed $c_{\rm brain} = 4.55$ and damping coefficient $a_{\rm brain} = 10.7$.

The source function $f$ is defined as a plane wave 
\begin{align} \label{Eqn.Source2}
f = - \left( \omega^2 \frac{(c_o^2 - c^2)}{c_o^2} (1 + i a_o/\omega) + i \omega (a - a_o) \right) u_{\rm inc}
\end{align}
where $\omega$ is the angular frequency of temporal oscillation, and the incident wave field $u_{\rm inc}$ is given by a Gaussian beam traveling downward (negative direction on the $x_2$-axis),
\begin{align} 
& u_{\rm inc} = \frac{W_o}{W(x_2 - x_{2,o})} \exp \left( - \left[\frac{x_1 - x_{1,o}}{W(x_2 - x_{2,o})} \right]^2 \right)  \exp \left( - i k \left[ x_2 - x_{2,o} + \frac{1}{2}\kappa(x_2)(x_1 - x_{1,o})^2  \right] \right)  \label{Eqn.GaussianBeam1} \\ & \text{where} \quad W(x_2) = W_o \sqrt{1 + \left(\frac{x_2}{x_{\rm R}}\right)^2}, \quad k = \frac{\omega}{c_o} \left[ 1 + i \frac{a_o}{\omega} \right]^{1/2}, \quad \text{and} \quad \kappa(x_2) = \frac{x_2}{ x_2^2 + x_{\rm R}^2}. \label{Eqn.GaussianBeam2}
\end{align}
Here, the point $(x_{1,o},x_{2,o})=(0, 0.3)$ is the focus of the beam and $W_o = 0.01$ is its waist, $c_o$ is the background wave speed, $a_o$ is the background damping coefficient, $\kappa(y)$ is the curvature of the beam, and $y_{\rm R} = \omega W_o^2 / (2 c_o)$ is called the Rayleigh range which is the point at which the curvature of the beam is the largest. The source $f$ given by \eqref{Eqn.Source2} in terms of the incident field $u_{\rm inc}$ defined by \eqref{Eqn.GaussianBeam1}-\eqref{Eqn.GaussianBeam2} corresponds to a scattering problem for an acoustic beam propagating in the background medium and impinging on the model of the human head. Therefore, the solution to \eqref{Eqn.001} for this source $f$ corresponds to an acoustic field scattered by the human head acting as a penetrable obstacle. The (real-valued) wave speed $c$ and the function $\zeta$ defining the absorbing layer are illustrated in Figure \ref{Fig.Ex2_Media}

\begin{figure}[H]
\centering
\captionsetup{font=small}
\subfloat[Real-valued wave speed $c$ ]{
\includegraphics[height=0.38 \textwidth, trim = 0 0 0 0, clip]{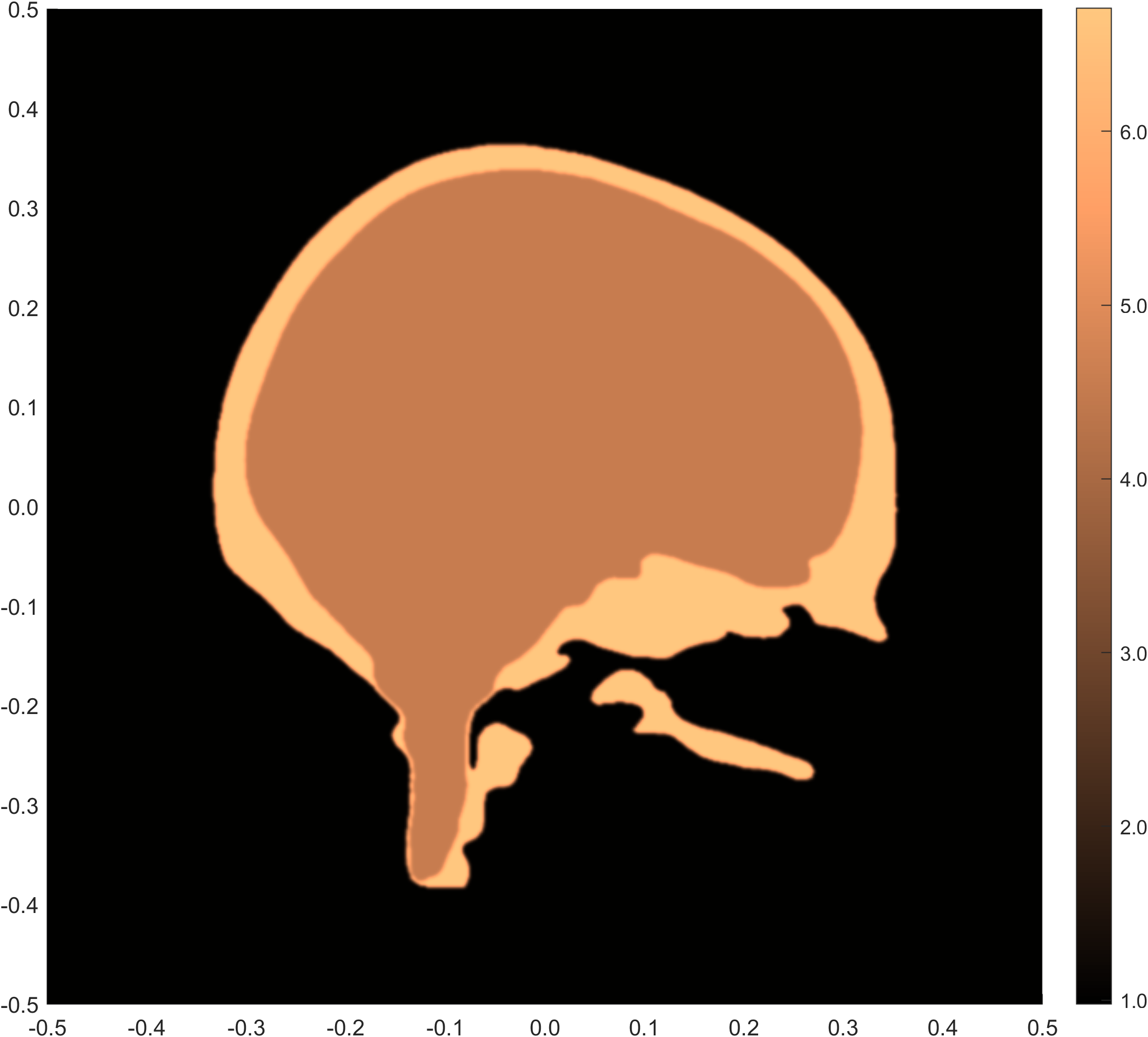}
\label{subfig:Real_c}
}
\subfloat[Function $\zeta$ for absorbing layer]{
\includegraphics[height=0.38 \textwidth, trim = 0 0 0 0, clip]{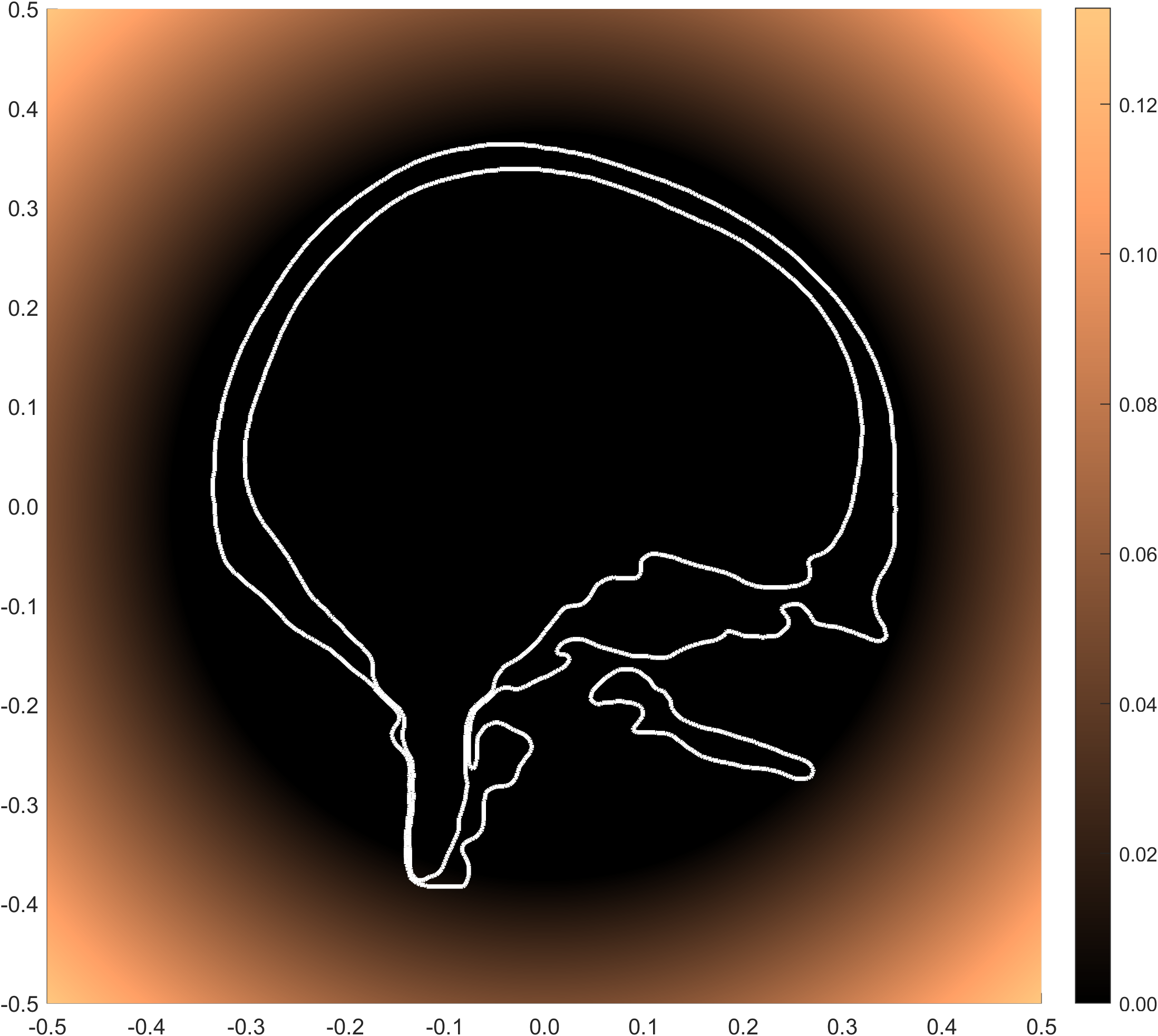}
\label{subfig:Zeta}
}
\caption{(a) Wave speed profile $c$ from the supplemental material of \cite{Aubry2022}. (b) Function $\zeta$ for the absorbing layer given by \eqref{Eqn.Zeta}. The skull contour is shown for reference. The complex-valued wave speed is defined by \eqref{Eqn.Complex_wave_speed}.}
\label{Fig.Ex2_Media}
\end{figure}

For the problem described by \eqref{Eqn.Zeta}-\eqref{Eqn.GaussianBeam2}, the evolution of the residuals over the GMRES iterations are illustrated in Figure \ref{Fig.Ex2} for frequencies $\omega = 400 \pi$, $\omega = 600 \pi$ and $\omega = 800 \pi$, which respectively corresponds to $200$, $300$ and $400$ wavelengths across the domain. In all cases, the the number of interpolation points is $M=4$ for the real values of the wave speed and $\tilde{M} = 3$ for the imaginary values of the wave speed (absorbing layer). The number of points per wavelength is PPW = $12$ which leads to $N=2400$, $N=3600$ and $N=4800$, respectively. Hence, DOF = $5.760 \times 10^6$, DOF = $1.296 \times 10^7$ and DOF = $2.304 \times 10^7$, respectively. As before, we used the MATLAB ``gmres'' function for matrix-free calculations using function handles. The restart parameter was set equal to $10$.

We observe from Figure \ref{Fig.Ex2} that applying the GMRES method naively to \eqref{Eqn.MainDiscrete} leads to a slow convergence pattern and that the application of the GMRES method to the preconditioned system \eqref{Eqn.MainDiscretePrecond} improves the convergence considerably. In fact, at $100$ iterations, the residual is at least $8$ orders of magnitude smaller when the proposed preconditioner is employed. For reference, the (log base 10) amplitude of the numerical solutions obtained after $100$ iterations are displayed in Figure \ref{Fig.Ex2}.

\begin{figure}[H]
\centering
\captionsetup{font=small}
\subfloat[Log-Amplitude for $\omega = 400 \pi$]{
\includegraphics[height=0.38 \textwidth, trim = 0 0 0 0, clip]{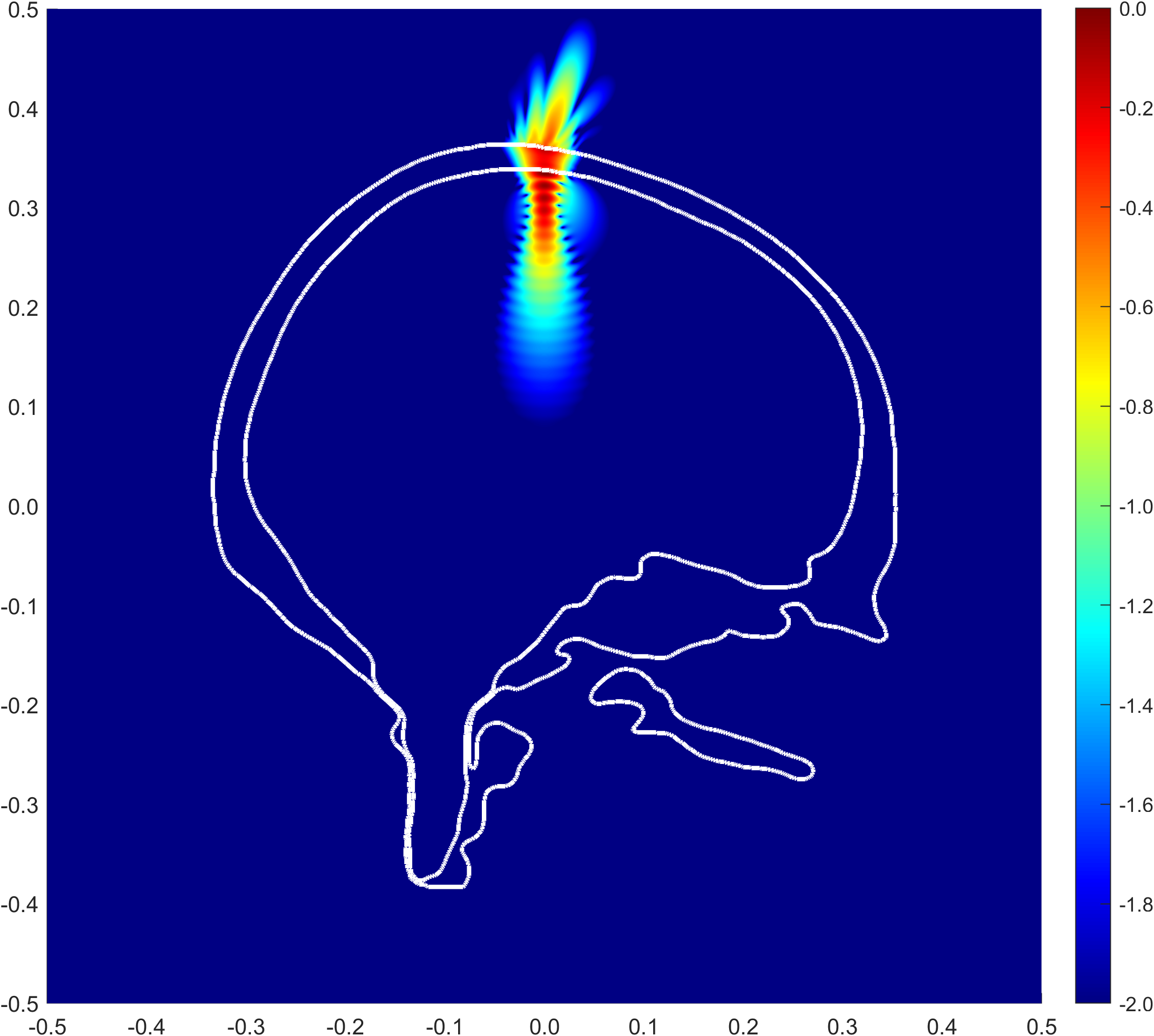}
\label{subfig:Ampl2_100}
}
\subfloat[Residuals for $\omega = 400 \pi$]{
\includegraphics[height=0.38 \textwidth, trim = 0 0 0 0, clip]{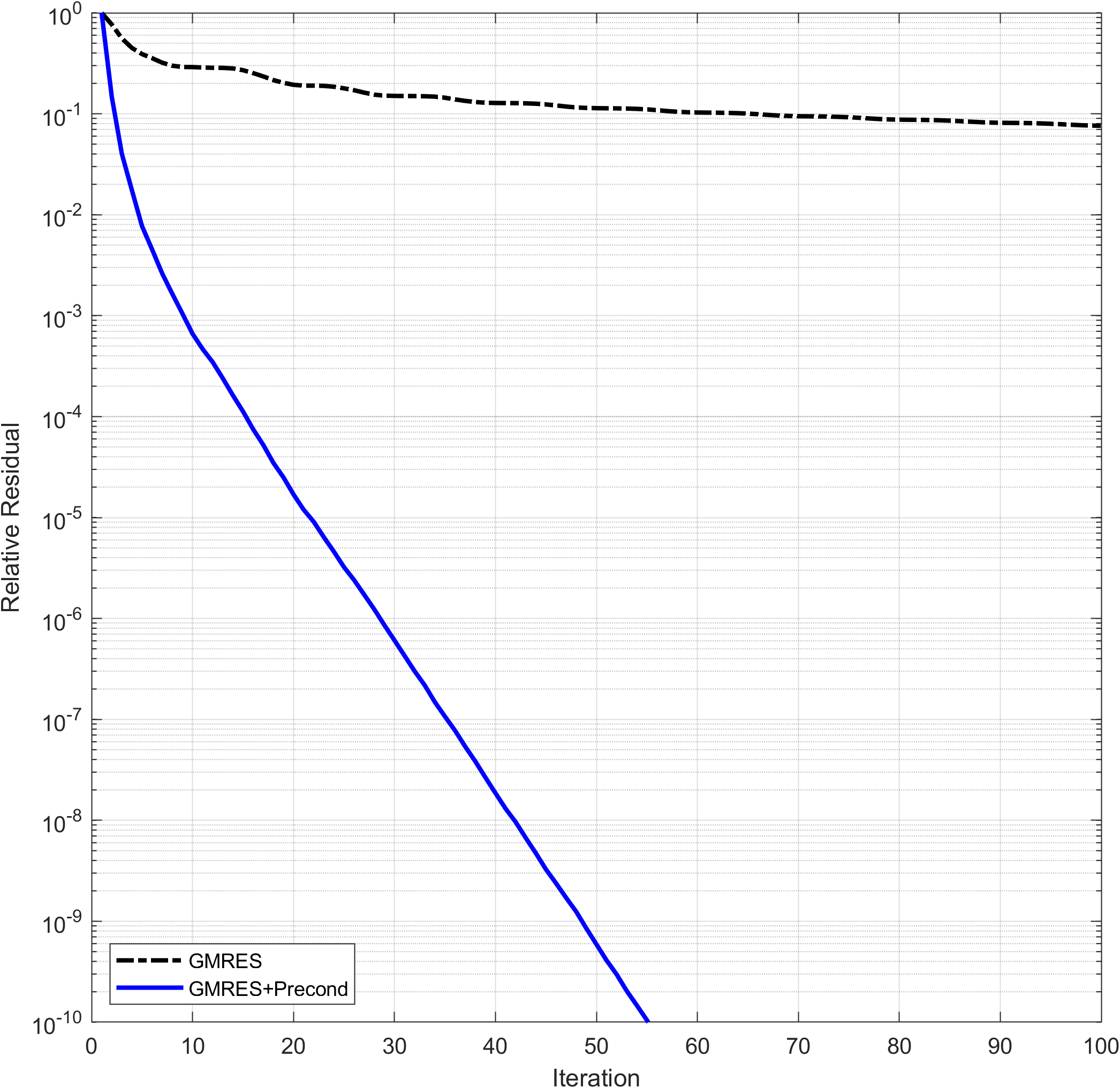}
\label{subfig:GMRES2_100}
} \\
\subfloat[Log-Amplitude for $\omega = 600 \pi$]{
\includegraphics[height=0.38 \textwidth, trim = 0 0 0 0, clip]{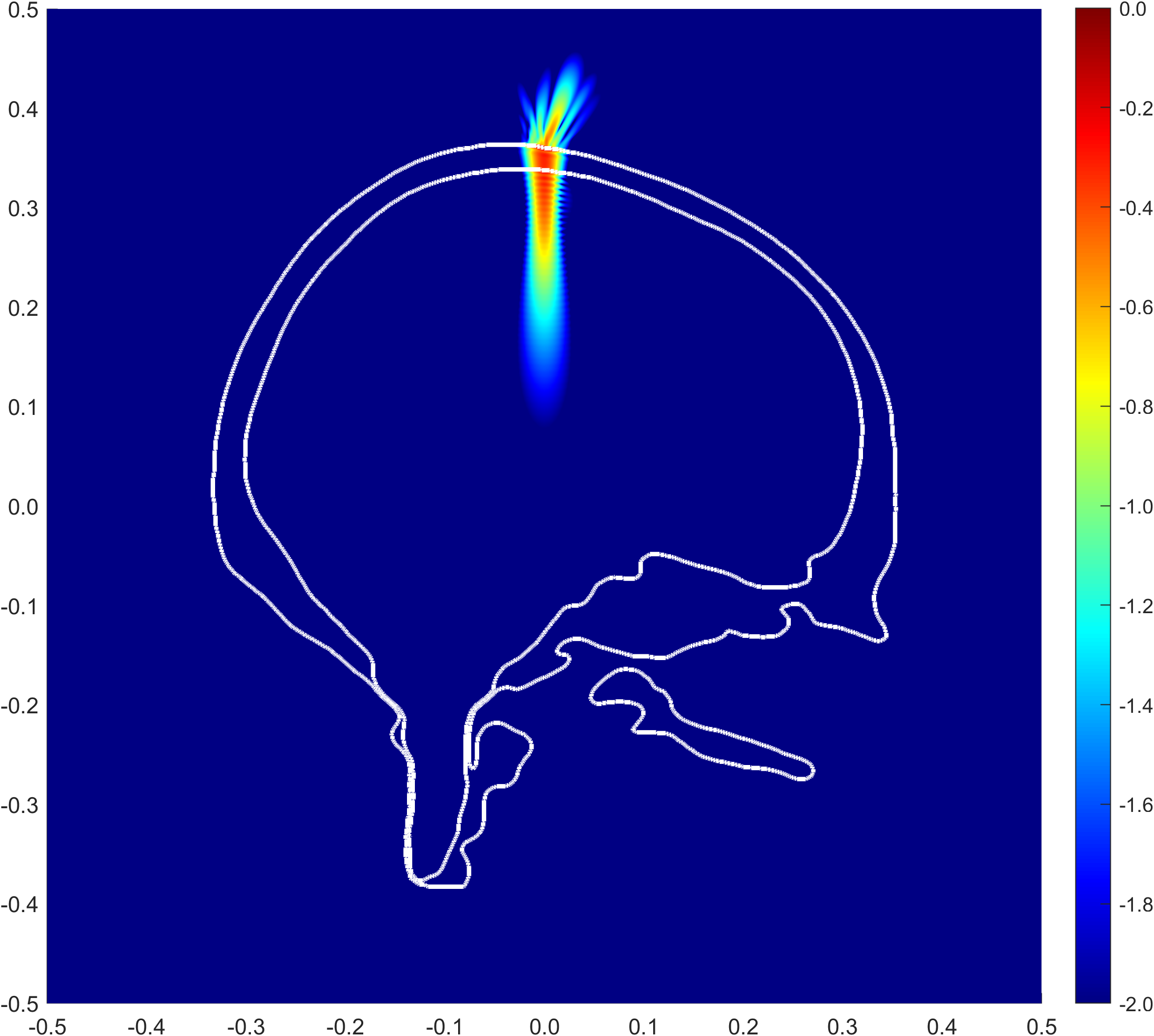}
\label{subfig:Ampl2_200}
}
\subfloat[Residuals for $\omega = 600 \pi$]{
\includegraphics[height=0.38 \textwidth, trim = 0 0 0 0, clip]{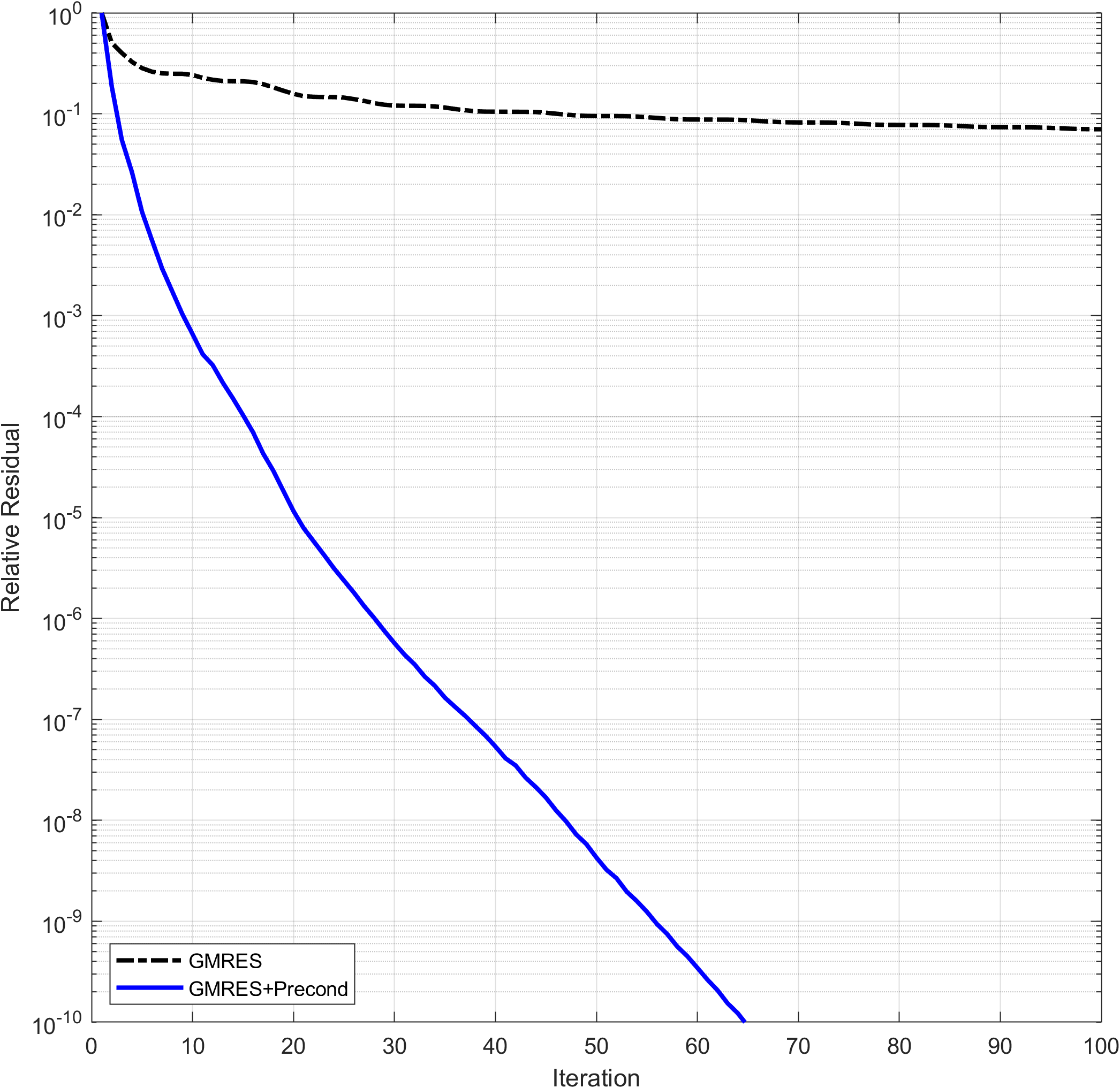}
\label{subfig:GMRES2_200}
} \\
\subfloat[Solution amplitude for $\omega = 800 \pi$ (log-scale)]{
\includegraphics[height=0.38 \textwidth, trim = 0 0 0 0, clip]{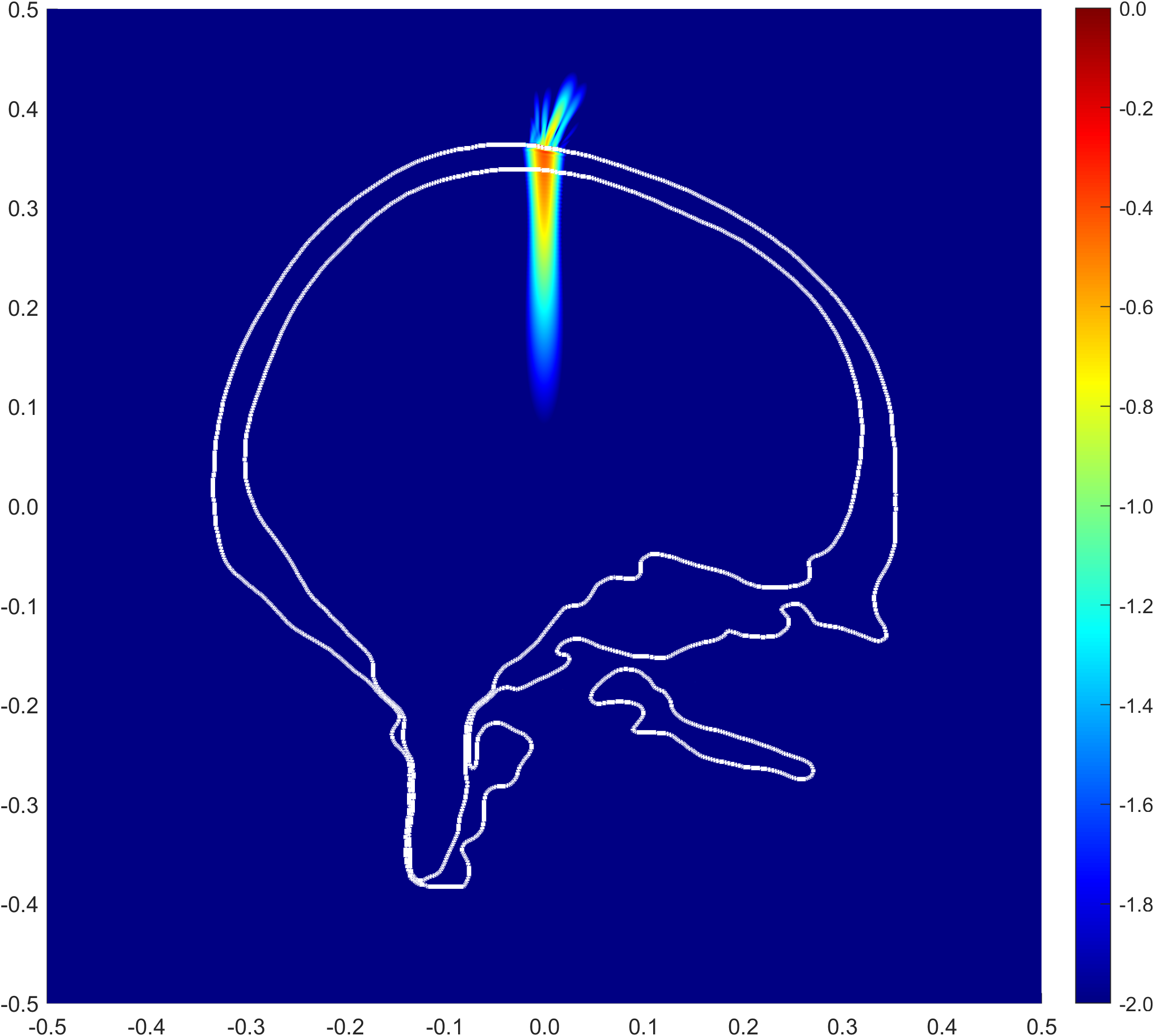}
\label{subfig:Ampl2_300}
}
\subfloat[Residuals for $\omega = 800 \pi$]{
\includegraphics[height=0.38 \textwidth, trim = 0 0 0 0, clip]{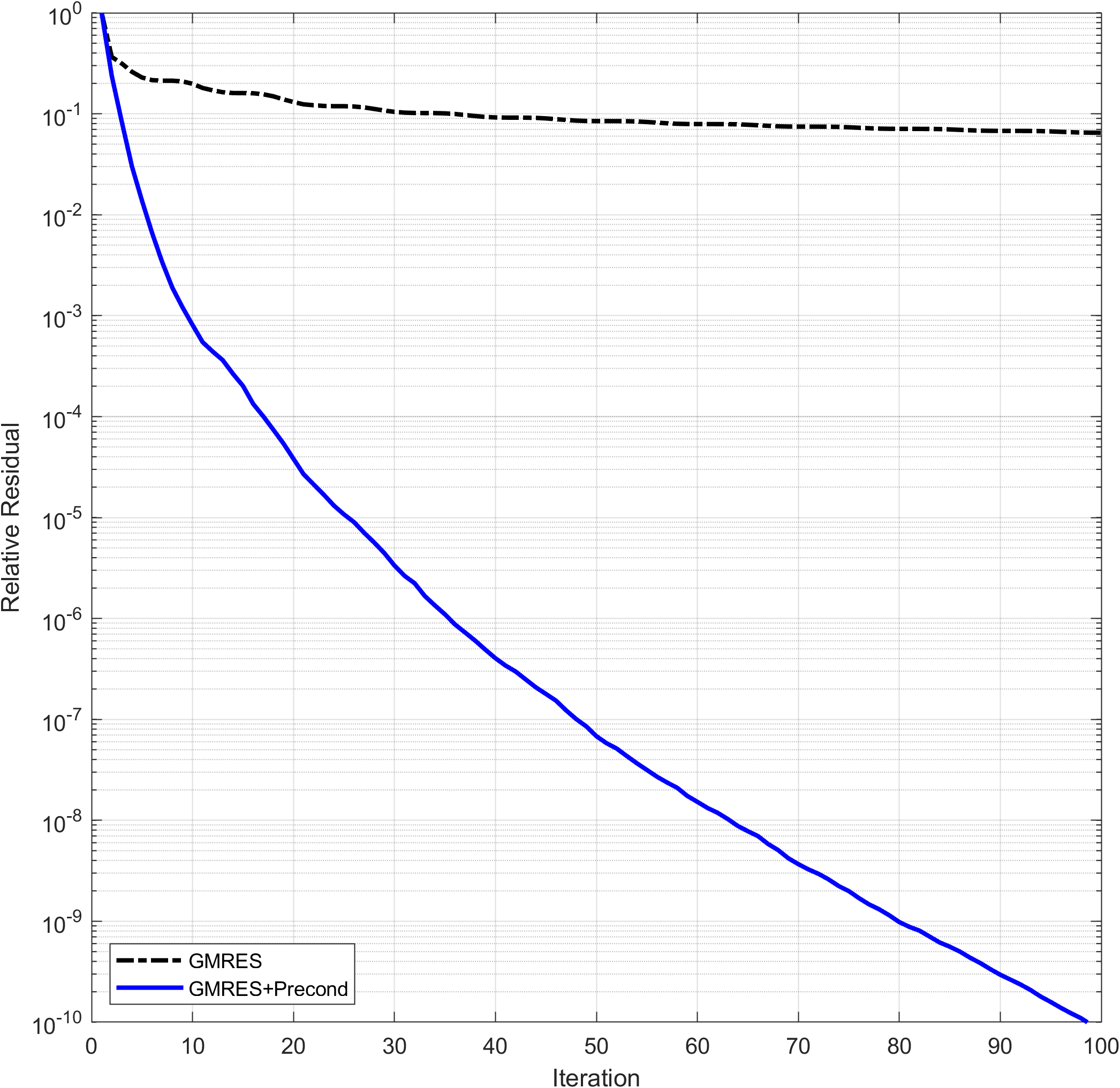}
\label{subfig:GMRES2_300}
}
\caption{Amplitude of the solution $u$ (log-scale) and relative residual for the first 100 iterations of the GMRES method applied to \eqref{Eqn.MainDiscrete} and to \eqref{Eqn.MainDiscretePrecond} for frequencies $\omega=400 \pi$, $\omega = 600 \pi$ and $\omega=800 \pi$, respectively.}
\label{Fig.Ex2}
\end{figure}

\bigskip

\bigskip

%%%%%%%%%%%%%%%%%%%%%%%%%%%%%%%%%%%
%%%%%%%%%  NEW SECTION  %%%%%%%%%%%
%%%%%%%%%%%%%%%%%%%%%%%%%%%%%%%%%%%

\section{Conclusion}
\label{Section.Conclusion}

In this paper, we proposed a new pseudodifferential preconditioner for the Helmholtz equation in variable media with absorption. This preconditioner is constructed from the principal symbol $q_{-2}=q_{-2}(x,\xi)$ (see \eqref{Eqn.021}) of a non-classical pseudodifferential expansion for the inverse of the Helmholtz operator. We take advantage of the explicit dependence of this principal symbol on the wave speed $c$ to interpolate it as a function of the wave speed $c$ rather than as a function of the phase-space variables $(x,\xi)$. As a result we obtain an accurate and computationally efficient approximation of this symbol using univariate interpolation theory even when the original wave problem is posed in two- or three-dimensional domains. The overall complexity of the proposed preconditioner is $\mathcal{O}(M N^{d} \log N)$ where $M$ is the number of interpolation points sampling the range of the wave speed, $N^d$ is the DOF, and $d=2,3$ is the dimension of the physical domain. The number $M$ of interpolation points is independent of the DOF, of the frequency $\omega$, and of the dimension $d$. The proposed preconditioner can be implemented as a matrix-free action based on Fourier transforms. The numerical experiments reported here show that the preconditioner can reduce the condition number of the discrete (finite difference) Helmholtz operator by 10 orders of magnitude or more (see Tables \ref{tab:tableFrequency}-\ref{tab:tableSmoothness}) and improve the convergence of the GMRES iteration considerably (see Figures \ref{Fig.Solution_u} and \ref{Fig.Ex2}). The numerical experiments also show that the conditioning properties of the proposed method remain robust with respect to the points per wavelength and to the frequency $\omega$. This latter observation is consistent with the pseudodifferential derivation whose neglected terms decay as the frequency increases. See remarks above Theorem \ref{Thm.Spectral1}. A second set of numerical experiments (see Subsection \ref{Section.Subsection2}) illustrate the ability of the preconditioner to be implemented for wave scattering problems where an absorbing boundary layer/sponge is needed to truncate the physical domain and approximate the effect of the Sommerfeld radiation condition.

Finally we wish to list some limitations of the proposed preconditioning method and directions of possible improvement or extensions.
\begin{enumerate}[leftmargin=2em]
    \item In general, the proposed pseudodifferential preconditioner is limited by its underlying technical assumptions, including the smoothness of the wave speed and other media properties. Although the principal symbol \eqref{Eqn.021} does not contain derivatives of the acoustic properties, the next symbol does. In fact, we have
    observed a deterioration in the effectiveness of the preconditioner in Table \ref{tab:tableSmoothness} as the transition in wave speed values becomes sharper.
    
    \item The domain of interest was discretized using a uniform Cartesian grid which is ideal for discretizing the Helmholtz operator using finite differences and for the implementation of the FFT for the application of the preconditioner. It would be interesting to investigate how to implement this preconditioner on non-Cartesian grids which are typically employed for finite element methods. 
    
    \item The speed of the proposed preconditioner relies on the application of the FFT. Unfortunately, FFT algorithms are not very well-suited for parallelization and distributed memory architectures. Hence, the FFT may prevent the proposed preconditioner from being fully parallelized.
    
    \item Assumption \eqref{Eqn.Intro02}, embedded in \eqref{Eqn.101}, may be too restrictive in some interesting cases. For instance, the wave speed $c=c(x)$ may be completely independent from the damping coefficient $a=a(x)$. In such cases, the interpolation method described in Section \ref{Section.EvalAlgo} should be extended to a bi-variate setting at the expense of increasing the number of interpolation points to sample the set $[c_{\rm min} , c_{\rm max}] \times [a_{\rm min} , a_{\rm max}]$. However, in many applications, such as ultrasonics in soft biological tissues, there is a one-to-one relationship between the wave speed and the attenuation coefficient of various tissues \cite{Varslot2005, Verweij2014}. Hence, the coefficient $a$ can be expressed as a function of the wave speed $c$. 

    \item In this paper, we have incorporated the presence of an absorbing layer/sponge defined by a complex-valued wave speed (see Subsections \ref{Subsec.AbsLayer} and \ref{Section.Subsection2}). However, perfectly matched layers (PML) have been shown to absorb the outgoing waves better. Hence, it would be interesting to extend the preconditioner to handle PMLs. 

    \item We have implemented the interpolation scheme using piecewise linear polynomials. Other interpolation bases such as piecewise Lagrange or Hermite polynomials could be implemented to improve the accuracy estimate of Theorem \ref{Thm.Acc}. In particular, Hermite interpolation may bring about interesting challenges as the derivative of a symbol with respect to the wave speed will introduce new symbols into the formulation.

    \item It may be possible to design more accurate preconditioners by including more terms from the expansion \eqref{Eqn.020}-\eqref{Eqn.025}. However, we anticipate some challenges. First, higher derivatives of the media properties (wave speed and damping coefficient) would need to be discretized. Second, a simple look at \eqref{Eqn.023} reveals that the symbol $q_{-3}$ represents a third order derivative in its numerator. This factor is eventually dominated by the denominator for very high frequencies $|\xi| \gg \omega/c$. But for frequencies in the range $0 \lesssim |\xi| \lesssim \omega / c$, we can expect the naive application of $q_{-3}$ to behave unstably. Therefore, regularization or appropriate filters may be needed.
    
    \item The Helmholtz equation is one of the simplest models for time-harmonic wave propagation. We believe that the work developed here could be extended to other models such as the Helmholtz equation is divergence form, Maxwell equations of electromagnetism, Navier equations of linear elasticity, and also to integral formulations such as the Lippmann-Schwinger equation.
\end{enumerate}

%%%%%%%%%%%%%%%%%%%%%%%%%%%%%%%%%%%
%%%%%%%%%  NEW SECTION  %%%%%%%%%%%
%%%%%%%%%%%%%%%%%%%%%%%%%%%%%%%%%%%
\bigskip

\section*{Acknowledgment}
The work of S. Acosta and T. Khajah was partially supported by NIH award 1R15EB035359-01A1. The work of B. Palacios was partially supported by Agencia Nacional de Investigaci\'on y Desarrollo (ANID) de Chile, Grant FONDECYT Iniciaci\'on N$^\circ$11220772.
S. Acosta would like to thank the support and research-oriented environment provided by Texas Children's Hospital. The authors would like to thank the anonymous journal referees for their constructive reviews.

\bigskip
%\clearpage

\bibliographystyle{plain}
\bibliography{library}

%\bibliographystyle{alpha}
%\bibliography{sample}

\end{document}